\documentclass[]{article}
\usepackage[a4paper,left=25mm,right=30mm,top=30mm,bottom=35mm]{geometry}
\usepackage{hyperref}
\usepackage{titlesec}
\usepackage{graphicx}
\usepackage{nomencl}
\makenomenclature
\usepackage{amsmath}%
\allowdisplaybreaks
\usepackage{color}
\usepackage{amsthm}
\usepackage{amsfonts}%
\usepackage{amssymb}%
\usepackage{graphicx}

\usepackage{paralist}
\usepackage[onehalfspacing]{setspace}
\usepackage{tikz}
\usepackage{color}
\usepackage{esint}
\usepackage{lmodern}

\theoremstyle{plain}

\newtheorem{definition}{Definition}[section]

\newtheorem{lemma}[definition]{Lemma}

\newtheorem{proposition}[definition]{Proposition}
\newtheorem{remark}[definition]{Remark}

\newtheorem{theorem}[definition]{Theorem}
\newtheorem{corollary}[definition]{Corollary}

\normalsize

\newcommand{\gast}{\nabla_{\ast}}

\newcommand{\gasta}[1]{\nabla_{\Gamma_{\ast}^{#1}}}

\newcommand{\Dast}{\Delta_{\ast}}
\newcommand{\Dasta}[1]{\Delta_{\Gamma_{\ast}^{#1}}}

\newcommand{\partialnuasta}[1]{\partial_{\nu_{\ast}^{#1}}}
\newcommand{\nuast}{\nu_{\ast}}

\newcommand{\intasti}{\int_{\Gammaast^i}}
\newcommand{\intsigmaast}{\int_{\Sigma_{\ast}}}
\newcommand{\brho}{\boldsymbol{\rho}}
\newcommand{\bmu}{\boldsymbol{\mu}}

\newcommand{\bu}{\boldsymbol{u}}

\DeclareMathOperator{\Id}{Id}
\DeclareMathOperator{\divergenz}{div}

\newcommand{\Area}{\text{Area}}
\newcommand{\Vol}{\text{Vol}}
\newcommand{\dH}{d\mathcal{H}}
\newcommand{\R}{\mathbb{R}}

\newcommand{\N}{\mathbb{N}}
\newcommand{\C}{\mathbb{C}}

\newcommand{\pr}{\text{pr}}
\newcommand{\supp}{\text{supp}}

\newcommand{\Gammaast}{\Gamma_{\ast}}

\newcommand{\GammaastT}{\Gamma_{\ast,T}}
\newcommand{\Sigmaast}{\Sigma_{\ast}}

\newcommand{\SigmaastT}{\Sigma_{\ast,T}}

\newcommand{\Gasti}{\Gamma_{\ast}^i}
\newcommand{\Nast}{N_{\ast}}
\newcommand{\tauast}{\tau_{\ast}}

\newcommand{\grhonull}{\gamma(\boldsymbol{\rho}_0)}

\definecolor{light-gray}{rgb}{0.23, 0.27, 0.29}

\title{Non-linear Stability of Double Bubbles under Surface Diffusion}
\date{}
\author{H. Garcke\footnote{University of Regensburg, 93040 Regensburg,
		harald.garcke@mathematik.uni-regensburg.de}, M. G\"o\ss{}wein\footnote{University of Regensburg, 93040 Regensburg,  michael.goesswein@mathematik.uni-regensburg.de}}
\begin{document}
	\setlength{\parindent}{0em} 
\maketitle
\begin{abstract}
We consider the evolution of triple junction clusters driven by the surface diffusion flow. On the triple line we use the boundary conditions derived by Garcke and Novick-Cohen as the singular limit of a Cahn-Hilliard equation with degenerated mobility. These conditions are the concurrency of the triple junction, angle conditions between the hypersurfaces, continuity of the chemical potentials and a flux-balance. For this system we show stability of its energy minimizers, i.e., standard double bubbles. The main argument relies on a \L ojasiewicz-Simon gradient inequality. The proof of it differs from others works due to the fully non-linear boundary conditions and problems with the (non-local) tangential part.
\end{abstract}

\section{Introduction}

\noindent
Motion by surface diffusion flow was firstly proposed by Mullins \cite{mullins1957theory} to describe the development of thermal grooves at grain boundaries of heated polycrystals. This process depends mainly on two physical effects, which are evaporation and surface diffusion, i.e., molecular motion on the surface of heated, solid substances. In the case that surface diffusion is the more involved process Mullins derived that the profile of the surface $\Gamma$ evolves due to
\begin{align}\label{EquationMotionDuetoSurfaceDiffusion}
V_{\Gamma}=-\Delta_{\Gamma}H_{\Gamma}.
\end{align} 
Hereby, $V_{\Gamma}$ denotes the normal velocity, $\Delta_{\Gamma}$ the Laplace-Beltrami operator and $H_{\Gamma}$ the mean curvature operator. For closed hypersurfaces a lot of research was done in the end of the last century. Cahn and Taylor identified the surface diffusion flow as formal $\mathcal{H}^{-1}$-gradient flow of the surface energy, cf. \cite{taylor1994linking}. The evolution is also connected to the Cahn-Hilliard equations as Cahn, Elliott and Novick-Cohen derived via formal asymptotics that the surface diffusion flow is its singular limit, cf. \cite{cahn1996cahn}. In the case of closed curves, short time existence was proven by Elliott and Garcke in \cite{elliottgarcke1994existence}. The result was generalized to closed hypersurfaces of arbitrary dimensions by Escher, Mayer and Simonett, cf. \cite{eschermayersimonnett1998surface}. Both works also show a stability result for spheres. Unlike other geometric flows, e.g., the mean curvature flow, it is very typical for the surface flow to have a loss of convexity and to develop singularities during the evolution, cf. \cite{gigaito1997pinching}, \cite{gigaito1999loss}. Chou studied in \cite{chou2003blopupcriterionclosedcurves} the behavior at singularities for planar closed curves evolving due to surface diffusion flow and derived a blow-up rate for the curvature energy in case of a singularity in finite time. The result was generalized to higher codimensions by Dziuk, Kuvert and Sch\"atzle, cf. \cite{dzuikkuvertschatzle2002Evolutionofelasticurves}. Wheeler proved in \cite{wheeler2012SurfaceDiffusionFlowNearSpheres} a long time existence result for initial data close to spheres. In contrast to the results in \cite{elliottgarcke1994existence} and \cite{eschermayersimonnett1998surface} closeness is hereby not measured in terms of the height function but by smallness of a curvature-like energy. Wheeler also gave criteria for the minimal existence time in \cite{wheeler2011lifespantheoremforsimpleconstrained}. As a general introduction to curvature driven flows and in particular for short time existence results and results on
the evolution of geometric quantities we refer to the book of Mantegazza \cite{mantegazzalecturesonmeancurvatureflow}. \\
In this article we are interested in the evolution of triple junction clusters under the surface diffusion flow. A triple junction cluster consists of a set of hypersurfaces such that each connected component of the boundary of a hypersurface is also boundary of two other hypersurfaces. We will restrict to the case of three connected hypersurfaces $\Gamma^1, \Gamma^2, \Gamma^3$ in $\R^n$ meeting in one triple junction $\Sigma$. The result transfers to more general situations. Additionally, we need that the $\Gamma^i$ are embedded, oriented, compact and do not intersect with each other. We will consider evolutions of such objects with respect to \eqref{EquationMotionDuetoSurfaceDiffusion} such that at every time $t$ we have a decomposition
\begin{align*}
\Gamma(t):=\Gamma^1(t)\cup \Gamma^2(t)\cup\Gamma^3(t)\cup \Sigma(t)
\end{align*} 
into the three hypersurfaces and the triple junction, such that the following boundary conditions are fulfilled:
\begin{align}
\partial\Gamma^1(t)=\partial\Gamma^2(t)&=\partial\Gamma^3(t)=\Sigma(t), \tag{CC}\\
\angle(\nu_{\Gamma^i(t)},\nu_{\Gamma^j(t)})&=\theta^k,\quad (i,j,k)\in \{(1,2,3),(2,3,1),(3,1,2)\},\tag{AC} \\
\gamma^1H_{\Gamma^1(t)}+\gamma^2H_{\Gamma^2(t)}+\gamma^3H_{\Gamma^3(t)}&=0, \tag{CCP}\\
\nabla_{\Gamma^1(t)}H_{\Gamma^1(t)}\cdot\nu_{\Gamma^1(t)}=\nabla_{\Gamma^2(t)}H_{\Gamma^2(t)}\cdot\nu_{\Gamma^2(t)}&=\nabla_{\Gamma^3(t)}H_{\Gamma^3(t)}\cdot\nu_{\Gamma^3(t)}. \tag{FB}
\end{align}
Here, $\nu_{\Gamma^i(t)}$ denotes the outer conormal of $\Gamma^i(t)$, $\gamma^1,\gamma^2,\gamma^3$ are constants determining the energy density on the hypersurfaces $\Gamma^i(t)$ and $\theta^1,\theta^2,\theta^3\in(0,\pi)$ are given angles, that are related to the $\gamma^i$. Indeed, the condition (AC) is equivalent to Young's law
\begin{align}
\frac{\sin(\theta^1)}{\gamma^1}=\frac{\sin(\theta^2)}{\gamma^2}=\frac{\sin(\theta^3)}{\gamma^3}.
\end{align}
The above problem has been introduced by Garcke and Novick-Cohen, cf. \cite{garcke2000singular}, and the condition (CCP) results from continuity of the chemical potentials at the triple junction and (FB) is equivalent to the flux balances. (CC) gives the concurrency of the triple junction during the flow.
\begin{figure}[h]\label{OneOfManyFigures}
	\centering
	\includegraphics[height=6cm]{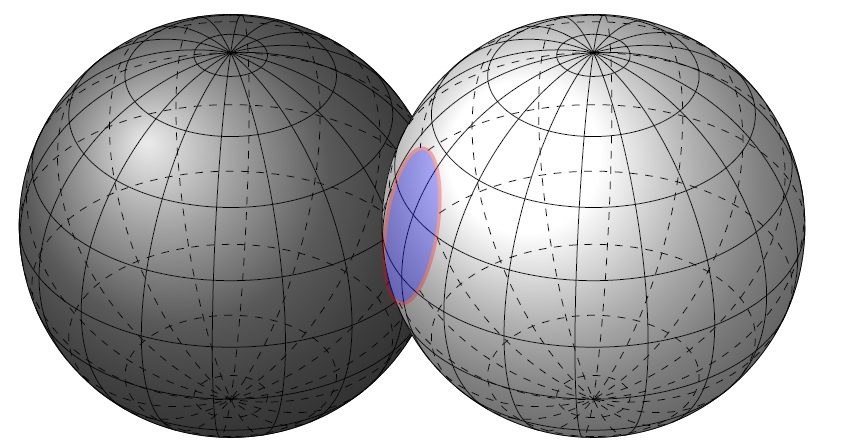}
	\caption{The picture shows the considered kind of triple junction cluster. In total, there are three hypersurfaces. In this illustration these are the two spherical caps and the flat blue area. The red line marks the triple junction, which is the boundary of all three hypersurfaces. Note that if the two enclosed volumes are unequal the blue surfaces will normally bend into direction of the larger volume.
	The shown geometry is the famous solution of the double-bubble conjecture, that is the geometric object with two fixed enclosed volumes and least surface area.}
\end{figure}\\
In this paper, we will show stability of standard double bubbles, see Figure \ref{OneOfManyFigures}, with respect to this evolution law. As it was shown in \cite{HutchingsMorganProofDoubleBubbleConjecture} these are precisely the minimizers of the surface area for two given enclosed volumes.  Roughly, we will prove the following.
\begin{theorem}[Stability of standard double bubbles]\ \\ 
	Let $\Gammaast\subset\R^n$ be a standard double bubble. Then $\Gammaast$ is Lyapunov stable in the following sense. There is for every $\varepsilon>0$ a $\delta>0$ such that for all triple junction clusters $\Gamma_0$ of class $C^{4+\alpha}$, which have a distance of at most $\delta$ from $\Gammaast$, the solution to the evolution with respect to surface diffusion flow \eqref{EquationMotionDuetoSurfaceDiffusion}, the boundary conditions (CC), (AC), (CCP), (FB), and initial data $\Gamma(0)=\Gamma_0$ exists globally in time and has at every time a maximum distance $\varepsilon$ from $\Gammaast$. Additionally, the solution converges to a (possibly different) standard double bubbles.
\end{theorem}

Hereby, the distance between triple junction clusters refers to the norm of the height function from the analytic framework we use to track the evolution. We will explain this in detail in Subsection \ref{SubsectionShortTimeExistence}. It is reasonable to identify this norm as a distance due the results from \cite{prusssimmonett2012manifoldofclosedhypersurfaces} for closed hypersurfaces.  
Our work is a generalization of the work of Abels, Arab and Garcke \cite{abelsarabgarcke2015standarddoublebubblearestableundersurfacediffusionflow}, where the authors showed the same result in the case of planar double bubbles. Their proof is based on a modification of the generalized principle of linearized stability. The main difficulty in its application lies in the fact that one needs an manifold structure for the set of critical points near the minimizer. In higher space or codimensions this is very difficult to verify and thus a lot of recent works use an approach with a so called \L ojasiewicz-Simon gradient inequality. With a standard argumentation, see, e.g., \cite[Section 4]{chillfavsangovaschatzle2009willmoreflowsarenevercompact}, this inequality can be used to do the desired stability analysis for problems with a gradient flow structure. To verify the inequality itself one can use a result from Chill \cite{Chill2003OnTheLojasiewiczSimon} that can be applied in a very general setting. For a concrete application the reformulation of \cite[Theorem 2]{FeehannMaridakisLojasiewiczSimon} Feehan and Meridakis is more comfortable as one only needs to check analyticity of the first and a Fredholm property of the second derivative. Due to this simplicity we consider this method a very powerful tool and want to give some examples on its applications. Abels and Wilke proved in \cite{abelswilke2007Convergencetoequilibiriumfor} stability of equilibria of the Cahn-Hilliard equations with a logarithmic free energy. In \cite{chillfavsangovaschatzle2009willmoreflowsarenevercompact}, Chill, Fasangova and Sch\"atzle considered evolution of closed hypersurfaces in higher codimensions due to the Willmore flow. Related to this is the work of Dall'Acqua, Pozzi and Spener \cite{dallacquapozzispener2016lojasiewichsimongradientclampedcurves}, where the authors considered clamped curves in higher codimensions driven by the Willmore flow. Lengeler studied in \cite{lengeler2018AsymptoticstabilityHelfrichminimizers} stability of local Helfrich minimizers under a Stokes type flow. In our situation some additional technical difficulties appear that were - to our best knowledge - not discussed yet. Firstly, the mentioned works restrict to surfaces of dimensions 1 or 2. In a higher dimensional situation one has to adapt the function spaces and cannot anymore work with the setting natural for the involved gradient flow structure. The reason for this is due to the fact that Sobolev spaces loose regularity properties when raising the space dimensions. Secondly, we have to discuss how to include the non-linear boundary conditions. And finally it turns out that the non-local tangential part induces term that cannot be dealt with. Thus, we have to replace it by a local version.\\
The article is organized as follows. In Section \ref{SectionPrelim} we will sum up the notation we use and then give an overview of important results for parabolic H\"older and Sobolev spaces on manifolds, especially taking the triple junction cluster into account. Additionally, we will state the results for the short time existence of the problem proven in a previous work. In Section \ref{SectionTechnicalAspects} we will give a general idea how to get a framework to prove the \L ojasiewicz-Simon gradient inequality. Afterwards, we will verify in Section \ref{SectionParabolicSmoothing} that the flow admits parabolic smoothing, at least if the initial data are close to standard double bubbles. In Section\ref{SectionTangentialPart} we will introduce a local tangential part and prove that this new tangential part describes the same set of triple junction clusters as the non-local version used in the existence analysis. Then, we will construct a parametrization of the non-linear volume constraints in Section \ref{SectionParametrization} and derive variational formulas for the parametrized surface energy in Section \ref{SectionVariationFormula}. With the work of the previous section we can prove the \L ojasiewicz-Simon gradient inequality itself in Section \ref{SectionProofofLsi}. Finally, we can use it to derive the desired stability results in Section \ref{SectionStabilityAnalysis}.

\section{Preliminaries}\label{SectionPrelim}
\noindent

\subsection{Notational Conventions}
\label{SectionNotation}
Throughout the whole article $\Gamma(t)$ will denote an evolving triple junction cluster at time $t$. We will denote by $\Gamma^1(t), \Gamma^2(t)$ and $\Gamma^3(t)$ the three hypersurfaces and by $\Sigma(t)$ the arising triple junction, that is
\begin{align*}
\Sigma(t)=\partial\Gamma^1(t)=\partial\Gamma^2(t)=\partial\Gamma^3(t).
\end{align*}
Hereby, we will only consider situations, where all $\Gamma^i(t)$ are compact, embedded, connected and do not intersect each other. Therefore, two of the hypersurfaces will always form a volume containing the third hypersurface, which we will choose to be $\Gamma^1$. By $\Omega_{12}$ and $\Omega_{13}$ we denote the volume enclosed by $\Gamma^1$ and $\Gamma^2$ resp. $\Gamma^1$ and $\Gamma^3$. We choose the normals $N_{\Gamma^i(t)}$ of the hypersurfaces such that the normal of $\Gamma_1$ points in the interior of $\Omega_{12}$, the one of $\Gamma^2$ outside of $\Omega_{12}$ and the one of $\Gamma^3$ into the inside of $\Omega_{13}$. The outer conormals will be denoted by $\nu$. \\
Furthermore, we use the standard notation for quantities of differential geometry, see for example \cite[Chapter 2]{kuehneldifferentialgeometry2015}. That includes the canonical basis $\{\partial_i\}_{i=1,...,n}$ of the tangent space $T_p\Gamma$ at a point $p\in\Gamma$ induced by a parametrization $\varphi$, the entries $g_{ij}$ of the metrical tensor $g$, the entries $g^{ij}$ of the inverse metric tensor $g^{-1}$, the Christoffel symbols $\Gamma^{i}_{jk}$ , the second fundamental form $II$, its squared norm $|II|^2$ and the entries $h_{ij}$ of the shape operator. We use the usual differential operators on a manifold $\Gamma$, which are the surface gradient $\nabla_{\Gamma}$, the surface divergence $\divergenz_{\Gamma}$ and the Laplace-Beltrami operator $\Delta_{\Gamma}$.\\
By $\rho$ we will denote the evolution in normal direction and by $\mu$ the evolution in tangential direction, which we will use to track the evolution of $\Gamma(t)$ over $\Gammaast$ via a direct mapping approach. $\Gamma_{\rho}$ resp. $\Gamma_{\rho,\mu}$ will denote the (triple junction) manifold that is given as graph over $\Gammaast$, cf. \eqref{EquationNormalGraphTJ}. Sub- and superscripts $\rho$ resp. $\mu$ on a quantity will indicate that the quantity refers to the manifold $\Gamma_{\rho,\mu}$. Hereby, we will normally omit $\mu$ as long it is given as function in $\rho$. An asterisk will denote an evaluation in the reference geometry. Both conventions are also used for differential operators. For example, we will write $\nabla_{\rho}$ for $\nabla_{\Gamma_{\rho}}$ and $\gast$ for $\nabla_{\Gammaast}$. We will denote by $J_{\rho}$ the transformation of the surface measure, that is,
\begin{align*}
\dH^n(\Gamma_{\rho})=J_{\rho}\dH^n(\Gammaast).
\end{align*} \\
If we index a domain or a submanifold in $\R^n$ with a $T$ or $\delta$ in the subscript, this indicates the corresponding parabolic set, e.g., $\Gamma_T=\Gamma\times[0,T]$. With an abuse of notation, in most parts of the work we will not differ between quantities on $\Gamma_{\rho,\mu}$ and the pullback of them on $\Gammaast$. In the parts dealing with triple junction manifolds the index $i$ will be used to indicate that a quantity refers to the hypersurface $\Gamma^i$. A quantity in bold characters will refer to the vector consisting of the quantity on the three hypersurfaces of a triple junction, e.g., $\boldsymbol{\rho}=(\rho^1,\rho^2,\rho^3)$.\\
For the used function spaces we want to clarify that a subscript $(0)$ denotes a function $\brho$ on a triple junction manifold $\Gamma$ the condition
\begin{align}\label{EquationLinearizedVolumeConstraint}
\int_{\Gamma^1}\rho^1\dH^n=\int_{\Gamma^2}\rho^2\dH^n=\int_{\Gamma^3}\rho^3\dH^n.
\end{align}
Basically, this is the linearization of the conversation of the two enclosed volumes $\Omega_{12}$ and $\Omega_{13}$. Also, we denote by $\fint_{\Gamma}f\dH^n$ the mean value of a function $f\in L^1(\Gamma)$, that is,
\begin{align*}
\fint_{\Gamma}f\dH^n:=\Area(\Gamma)^{-1}\int_{\Gamma}f\dH^n.
\end{align*}
The subscript TJ in a function space will indicate that the function space has to be read as product space on each hypersurface. For example, we write
\begin{align*}
L^2_{TJ}(\Gamma):=L^2(\Gamma^1)\times L^2(\Gamma^2)\times L^2(\Gamma^3).
\end{align*} 
Although we work with the results of \cite{FeehannMaridakisLojasiewiczSimon}, we follow the notation of \cite{Chill2003OnTheLojasiewiczSimon}. In particular, $E: V\to \R$ denotes an energy functional (in our case the surface area) on a Banach spaces $V$, $\mathcal{M}$ its first derivative and $\mathcal{L}(0)$ its second derivative at point $0$. Hereby, we will always consider $\mathcal{L}(0)$ as function on $V$ with values (on a subset of) $V'$.   \\
Finally, we will always use the convention of dynamical constants. 

\subsection{Some Important Results from Functional Analysis}

A very important tool for our analysis will be the implicit function theorem for maps between Banach spaces. Therefore, we want to mention the following version from \cite[Theorem 4B]{ZeidlerNonlinearFunctionalAnalysisI}.
\begin{proposition}[Implicit function theorem of Hildebrandt and Graves]\label{PropostionImplicitFunction}\ \\
	Suppose that:\vspace{-0,1cm}
	\begin{itemize}
		\item[i.)] the mapping $F\colon U(x_0,y_0)\subset X\times Y\to Z$ is defined on an open neighborhood $U(x_0,y_0)$, and $F(x_0,y_0)=0$, where $X,Y,Z$ are Banach spaces over $\mathbb{K}\in\{\R,\C\}$ and $x_0\in X, y_0\in Y$. \vspace{-0,2cm}
		\item[ii.)] $\partial_y F$ exists as a partial Fr\'echet-derivative on $U(x_0,y_0)$ and $\partial_yF(x_0,y_0)\colon Y\to Z$ is bijective. \vspace{-0,2cm}
		\item[iii.)] $F$ and $\partial_y F$ are continuous at $(x_0,y_0)$.\vspace{-0,2cm}
	\end{itemize}
	Then, the following are true:\vspace{-0,2cm}
	\begin{itemize}
		\item[a.)] Existence and uniqueness: There exist positive numbers $r_0$ and $r$ such that for every $x\in X$ satisfying $\|x-x_0\|_X\le r_0$ there is exactly one $y(x)\in Y$ for which $\|y(x)-y_0\|_Y\le r$ and $F(x,y(x))=0$.\vspace{-0,2cm}
		\item[b.)] Continuity: If $F$ is continuous in a neighborhood of $(x_0,y_0)$, then $y(\cdot)$ is continuous in a neighborhood of $x_0$.\vspace{-0,2cm}
		\item[c.)] Continuous differentiability: If $F$ is a $C^m$-map on a neighborhood of $(x_0,y_0)$, $1\le m\le \infty$, then $y(.)$ is also a $C^m$-map on a neighborhood of $x_0$.
	\end{itemize}
\end{proposition}
We will need to study analyticity of maps between Banach spaces. To define it, we first have to introduce so called power operators.
\begin{definition}[Power operator]\ \\
	Let $X$ and $Y$ be Banach spaces over $\mathbb{K}\in\{\R,\C\}$. Let there be given a $k$-linear, bounded operator $M: X\times\cdots\times X\to Y$ which is symmetric in all variables. A power operator of degree $k$ is created from setting for all $m,n\in\{0,1,...,k\}$ with $m+n=k$ and $x,y\in X$
	\begin{align*}
	Mx^my^n:=M(\underbrace{x,...,x}_{m\text{ times}},\underbrace{y,...,y}_{n-\text{times}}).
	\end{align*}
\end{definition}
\begin{definition}[Analytic Operators between Banach Spaces]\ \\
	Let $Z$ and $Y$ be Banach spaces over $\mathbb{K}$ and $T: Z\supset D(T)\to Y$ defined on an open set $D(T)$. \vspace{-0,2cm}
	\begin{itemize}
		\item[i.)] $T$ is called analytic at $z_0\in D(T)$, if there is a sequence $\{T_k\}_{k\in\N_0}$ of power operators of degree $k$ together with an open neighborhood $U$ of $z_0\in D(T)$ such that for all $z\in U$ the series
		\begin{align}
		Sz:=\sum_{k=0}^{\infty} T_k(z-z_0)^k
		\end{align}   
		exists and we have $Sz=Tz$ for all $z\in U$. \vspace{-0,3cm}
		\item[ii.)] $T$ is called analytic on an open subset $V\subset D(T)$, if it is analytic at every point $z_0\in V$.
	\end{itemize}
\end{definition}
\begin{remark}
	Note that if $T$ is analytic at a point $z_0$, this implies that $T$ is analytic in an open neighborhood of $z_0$, cf. \cite[p.98]{ZeidlerNonlinearFunctionalAnalysisI}.
\end{remark}
Very important for our work will also be that the implicit function theorem inherits also analyticity, which is Corollary 4.23 in \cite{ZeidlerNonlinearFunctionalAnalysisI}.
\begin{corollary}[Analytic version of the implicit function theorem]\label{AnalyticVersionOfImplicitFuntion}\ \\
	If in the situation of Proposition \ref{PropostionImplicitFunction} the function $F$ is also analytic at $(x_0,y_0)$, then the solution $y$ is analytic at $x_0$ as well.
\end{corollary}
Finally, we note the following fact about compact perturbations of Fredholm operators, which is Proposition 8.14(3) in \cite{ZeidlerNonlinearFunctionalAnalysisI}.
\begin{proposition}[Compact perturbation of Fredholm operators]\label{PropostionCompactPerturbationofFredhom}\ \\
	Let $X,Y$ be Banach spaces, $B:X\to Y$ a Fredholm operator and $C: X\to Y$ a compact operator. Then the sum $B+C: X\to Y$ is also a Fredholm operator and the Fredholm index satisfies
	\begin{align}
	ind(B+C)=ind(B).
	\end{align}  
\end{proposition}

\subsection{Function Spaces on Manifolds}
In this section we want to introduce the two most important function spaces on manifolds we will use. These are Sobolev and parabolic H\"older spaces. In this section, $(\Gamma,\mathcal{A})$ will always be a compact, orientable, embedded submanifold $\Gamma$ of $\R^{n+1}$ with boundary together with a maximal atlas $\mathcal{A}$. 
\begin{definition}[Sobolev spaces on manifolds]
	Let $\Gamma$ be of class $C^j, j\in\N$. Then we define for $k\in\N, k<j, 1\le p\le \infty$ the Sobolev space $W^{k,p}(\Gamma)$ as the set of all functions $f: \Gamma\to\R$, such that for any chart $\varphi\in\mathcal{A}, \varphi: V\to U$ with $V\subset\Gamma, U\subset \R^n$ the map $f\circ \varphi^{-1}$ is in $W^{k,p}(U)$. Hereby, $W^{k,p}(U)$ denotes the usual Sobolev space. We define a norm on $W^{k,p}(\Gamma)$ by
	\begin{align}\label{EquationNormSobolevSpace}
	\|f\|_{W^{k,p}(\Gamma)}:=\sum_{i=1}^s \|f\circ \varphi_i^{-1}\|_{W^{k,p}(U_i)},
	\end{align}
	where $\{\varphi_i: V_i\to U_i\}_{i=1,...,s}\subset\mathcal{A}$ is a family of charts that covers $\Gamma$. 
\end{definition}
\begin{remark}[Equivalent norms on $W^{k,p}(\Gamma)$]\ \\\vspace{-0,6cm}
	\begin{itemize}
		\item[i.)] The norm on $W^{k,p}(\Gamma)$ depends on the choice of the $\varphi_i$ but for a different choice we will get an equivalent norm as the transitions maps are $C^j$. \vspace{-0,2cm}
		\item[ii.)] For the space $W^{1,p}(\Gamma)$ we will use the norm
		\begin{align}
		\|f\|_{W^{1,p}(\Gamma)}=\left(\int_{\Gamma}|\nabla_{\Gamma}f|^p+|f|^p\dH^n\right)^{\frac{1}{p}}.
		\end{align}
		Equivalence to (\ref{EquationNormSobolevSpace}) follows directly from the representation of the surface gradient in local coordinates. \vspace{-0,2cm}
		\item[iii.)] As usual we will write $H^k(\Gamma)$ for $W^{k,2}(\Gamma)$.
	\end{itemize}
\end{remark}
We want to make some further remarks on three properties of these spaces. The first one is a sufficient condition such that we get a Banach algebra structure.
\begin{lemma}[Banach space property of Sobolev spaces]\label{LemmaBanachspacepropertieofWkp}\ \\ 
	Let $\Gamma$ be smooth, $k\in\N, 1\le p\le\infty$, and assume that
	\begin{align}\label{EquationConditionforBanachAlgebra}
	p>\frac{2n}{k}.
	\end{align} 
	Then, $W^{k,p}(\Gamma)$ is a Banach algebra. In particular, $H^k(\Gamma)$ is a Banach algebra for $k>n$.
\end{lemma}
\begin{proof}
Cf. \cite[Lemma 2.9]{goesswein2019Dissertation}.
\end{proof}
For a differentiable manifold the Poincar\'e inequality is well known, cf. \cite[Theorem 2.10]{hanzawaclassicalsolutionstefanproblem}. This can also be used for each surface of $\Gamma$. But by imposing additional boundary conditions, namely \eqref{EquationLinearizedVolumeConstraint}, one also can guarantee a version for the whole cluster.
\begin{proposition}[Poincar\'e-type inequality on triple junction manifolds]\label{LemmaPoincareineuqlaityfortriplejunction}\ \\
	Let $\gamma^i>0, i=1,2,3$. Consider the space 
	\begin{align}\label{EquationSpaceH10ontripejunctions}
	\mathcal{E}:=\left\{\boldsymbol{\rho}\in H^{1}_{TJ}(\Gamma)\Bigg| \sum_{i=1}^3\gamma^i\rho^i=0, \int_{\Gamma^1}\rho^1\dH^n=\int_{\Gamma^2}\rho^2\dH^n=\int_{\Gamma^3}\rho^3\dH^n\right\}.
	\end{align}
	Then, there is a constant $C>0$ such that for all $\boldsymbol{\rho}\in\mathcal{E}$ we have
	\begin{align}\label{EquationPoincareInequalityHeinsTripleJunction}
	\|\boldsymbol{\rho}\|_{L^2_{TJ}(\Gamma)}\le C\|\nabla_{\Gamma}\boldsymbol{\rho}\|_{L^2_{TJ}(\Gamma)}.
	\end{align}
\end{proposition}
\begin{proof}
	Cf. \cite[Proposition 2.10]{goesswein2019Dissertation}.
\end{proof}
The last property we want to mention concerning Sobolev spaces is the space $\mathcal{H}^{-1}$, which is the dual space of $\mathcal{E}$ from (\ref{EquationSpaceH10ontripejunctions}). As we showed above we have a Poincar\'e inequality on both spaces and therefore an equivalent inner product on these spaces is given by the $L^2$-product of the surface gradients. Using the Riesz isomorphism we can identify the elements of $\mathcal{H}^{-1}$ with $\mathcal{E}$, that is, for every $f\in\mathcal{H}^{-1}$ there is a unique $\rho\in\mathcal{E}$ with
\begin{align}
\int_{\Gamma}\nabla_{\Gamma}\rho\cdot\nabla_{\Gamma}\psi\dH^n=f(\psi)\quad \forall\psi\in \mathcal{E}.
\end{align}
But this is nothing else but the weak formulation of
\begin{align}
-\Delta_{\Gamma^i}\rho^i&=f^i & & \text{on }\Gammaast^i, i=1,2,3,\\
\gamma^1\rho^1+\gamma^2\rho^2+\gamma^3\rho^3&=0 & &\text{on }\Sigma,\\
\partial_{\nu^1}\rho^1=\partial_{\nu^2}\rho^2&=\partial_{\nu^3}\rho^3 & &\text{on }\Sigma.
\end{align}
Therefore, we will write for the element from the Riesz identification $(-\Delta_{\Gamma})^{-1}f$ and get the inner product on $\mathcal{H}^{-1}$ by
\begin{align}
\langle f,g\rangle_{\mathcal{H}^{-1}}:=\int_{\Gamma}\nabla_{\Gamma}((-\Delta_{\Gamma})^{-1}f)\cdot \nabla_{\Gamma}((-\Delta_{\Gamma})^{-1}g)\dH^n, \quad f,g\in\mathcal{H}^{-1}.
\end{align}
We will later need the following interpolation result.
\begin{lemma}[Interpolation between $\mathcal{H}^{-1}$ and $H^1$]\label{LemmaInterpolationH-1H1}\ \\
Let $\boldsymbol{\rho}\in\mathcal{E}$. Then, we have that 
	\begin{align}
	\|\boldsymbol{\rho}\|^2_{L^2_{TJ}(\Gamma)}\le  \|\boldsymbol{\rho}\|_{\mathcal{H}^{-1}(\Gamma)}\|\boldsymbol{\rho}\|_{H^1_{TJ}(\Gamma)}.
	\end{align}
\end{lemma}
\begin{proof}
Cf. \cite[Lemma 2.12]{goesswein2019Dissertation}.
\end{proof} 
\begin{remark}[Different dual spaces]\ \\
	We will also need the dual space of $H^{1}_{TJ}(\Gammaast)$ equipped with the normal Sobolev norm. We will denote this space by $H^{-1}_{TJ}(\Gammaast)$. Clearly, the interpolation result from Lemma \ref{LemmaInterpolationH-1H1} is also true for $H^{-1}_{TJ}(\Gammaast)$.
\end{remark}
Now, we want to move on to the second kind of important function spaces, the parabolic H\"older spaces. It is both possible to introduce these spaces on manifolds in local coordinates, e.g. \cite[p.177]{lunardi2012analytic}, or without, e.g. \cite{depner2014mean}. We prefer the first approach as we want to use local results. We will first introduce these spaces on a bounded domain $\Omega$ in $\R^{n}$ with smooth boundary $\partial\Omega$. For this, we first need for $\alpha\in(0,1), a,b\in\R$ the two semi-norms for a function $f: \bar{\Omega}\times[a,b]\to\R$ given by\label{PageH-1produkt}
\begin{align*}
\langle f\rangle_{x,\alpha}&:=\sup_{x_1,x_2\in\bar{\Omega},t\in[a,b]}\frac{|f(x_1,t)-f(x_2,t)|}{|x_1-x_2|^{\alpha}},\\
\langle f\rangle_{t,\alpha}&:=\sup_{x\in\bar{\Omega},t_1,t_2\in[a,b]}\frac{|f(x,t_1)-f(x,t_2)|}{|t_1-t_2|^{\alpha}}.
\end{align*}
Now, we define for $k, k'\in\N, \alpha\in (0,1), m\in\N$ the spaces
\begin{align*}
C^{\alpha,0}(\bar{\Omega}\times[a,b])&:=\{f\in C(\bar{\Omega}\times[a,b])|\langle f\rangle_{x,\alpha}<\infty\},\\
\|f\|_{C^{\alpha,0}(\bar{\Omega}\times[a,b])}&:=\|f\|_{\infty}+\langle f\rangle_{x,\alpha},\\
C^{0,\alpha}(\bar{\Omega}\times[a,b])&:=\{f\in C(\bar{\Omega}\times[a,b])| \langle f\rangle_{t,\alpha}<\infty\},\\
\|f\|_{C^{0,\alpha}(\bar{\Omega}\times[a,b])}&:=\|f\|_{\infty}+\langle f\rangle_{t,\alpha}, \\
C^{k+\alpha, 0}(\bar{\Omega}\times[a,b])&:=\{f\in C(\bar{\Omega}\times[a,b])|\forall t\in[a,b]: f\in C^k(\bar{\Omega}),\\
&\phantom{:=\{ }  \forall\beta\in \N^n_0, |\beta|\le k: \partial^x_{\beta}f\in C^{\alpha,0}(\bar{\Omega}\times[a,b])\},\\
\|f\|_{C^{k+\alpha, 0}(\bar{\Omega}\times[a,b])}&:=\sum_{|\beta|\le k}\|\partial_{\beta}^xf\|_{\infty}+\sum_{|\beta|=k}\langle \partial_{\beta}^xf \rangle_{x,\alpha}, \\
C^{k+\alpha,\frac{k+\alpha}{m}}(\bar{\Omega}\times[a,b])&:=\{f\in C(\bar{\Omega}\times[a,b])|\forall\beta\in\N_0^n,i\in\N_0,mi+|\beta|\le k:\ \\ &\phantom{:=\{ }\partial_t^i\partial_{\beta}^xf\in C^{\alpha,0}(\bar{\Omega}\times[a,b])\cap C^{0,\frac{k+\alpha-mi-|\beta|}{m}}(\bar{\Omega}\times[a,b])  \},\\
\|f\|_{C^{k+\alpha,\frac{k+\alpha}{m}}(\bar{\Omega}\times[a,b])}&:=\sum_{0\le mi+|\beta|\le k}\left(\|\partial_t^i\partial_{\beta}^xf\|_{\infty}+\|\partial_t^i\partial_{\beta}^xf\|_{C^{0,\frac{k+\alpha-mi-|\beta|}{m}}(\bar{\Omega}\times[a,b])}\right)\\
&+\sum_{mi+|\beta|=k}\|\partial_t^i\partial_{\beta}^xf\|_{C^{\alpha,0}(\bar{\Omega}\times[a,b])}.
\end{align*}
Hereby, we denote by $\partial^x_{\beta}$ a partial derivative in space with respect to the multi-index $\beta$ and $\partial^i_t$ the $i$-th partial derivative in time. The parameter $m$ corresponds to the order of the differential equation one is considering and in our work it will always be four. Now, we can also define parabolic H\"older spaces on submanifolds as follows.
\begin{definition}[Parabolic H\"older spaces on submanifolds]\label{DefinitionParabolichHspacesmfd}\ \\
	Let $\Gamma$ be a $C^{r}$-submanifold of $\R^n$, either with or without boundary. Then we define for $k\in\N_0, k<r, \alpha\in(0,1), a,b\in\R,m \in\N$ the space $C^{k+\alpha,\frac{k+\alpha}{m}}(\Gamma\times[a,b])$ as the set of all functions $f:\Gamma\to\R$ such that for any parametrization $\varphi:\Omega\to V\subset\Gamma$ we have that $f\circ\varphi\in C^{k+\alpha,\frac{k+\alpha}{m}}(\bar{\Omega}\times[a,b])$. 
\end{definition}
\begin{remark}[Traces of parabolic H\"older spaces]\label{RemarkTracesParabolicHspaces}\ \\
	On the boundary $\Sigma$ of $\Gamma$ we may choose $\varphi$ to be a parametrization that flattens the boundary. From this we see that
	\begin{align*}
	f\in C^{k+\alpha,k'+\alpha'}(\Gamma\times[a,b])\Rightarrow f\big|_{\Sigma\times[a,b]}\in C^{k+\alpha,k'+\alpha'}(\Sigma\times[a,b]).
	\end{align*}
\end{remark}
\begin{remark}[H\"older regularity in time for derivatives]\ \\
	In some works these spaces are introduced with lower H\"older regularity in time for the lower order derivatives, cf. \cite{lunardi2012analytic} and \cite{depner2014mean}. Actually, this approach is equivalent due to interpolation results for H\"older continuous functions, cf. \cite[Proposition 1.1.4 and 1.1.5]{lunardi2012analytic}.
\end{remark}
Additionally, we want to mention two strong properties of parabolic H\"older spaces: its product estimates and the contractivity of lower order terms.
\begin{lemma}[Product estimates in parabolic H\"older spaces]\label{LemmaProductEstimatesParabolichHspace}\ \\
	Let $k,m\in\N,\alpha\in(0,1)$ and $f,g\in C^{k+\alpha,\frac{k+\alpha}{m}}(\overline{\Omega}\times[0,T])$. Then we have 
	\begin{align}\label{EquationProductareClosedinHolderSpaces}
	fg\in C^{k+\alpha,\frac{k+\alpha}{m}}(\overline{\Omega}\times[0,T]),
	\end{align}
	and furthermore we have that
	\begin{align}
	\|fg\|_{C^{k+\alpha,\frac{k+\alpha}{m}}(\overline{\Omega}\times[0,T])}&\le C \|f\|_{C^{k+\alpha,\frac{k+\alpha}{m}}(\overline{\Omega}\times[0,T])}\|g\|_{C^{k+\alpha,\frac{k+\alpha}{m}}(\overline{\Omega}\times[0,T])},\label{EquationProductEstimatesHolderSpaces1}\\
	\|fg\|_{C^{k+\alpha,\frac{k+\alpha}{m}}(\overline{\Omega}\times[0,T])}&\le C\left(\|f\|_{C^{k+\alpha,\frac{k+\alpha}{m}}(\overline{\Omega}\times[0,T])}\|g\|_{C^{k,0}}+\|f\|_{C^{k,0}}\|g\|_{C^{k+\alpha,\frac{k+\alpha}{m}}(\overline{\Omega}\times[0,T])}\right).\label{EquationProductEstimatesHolderSpaces2}
	\end{align}
\end{lemma}
\begin{proof}
Cf. \cite[Lemma 2.16]{goesswein2019Dissertation}.
\end{proof}
\begin{lemma}[Contractivity property of lower order terms in parabolic H\"older spaces]\label{LemmaContractivityLowerOrderTerms}\ \\
	Let $\Omega\subset\R^n$ be a bounded domain with smooth boundary and $$k,k'\in\{0,1,2,3,4\},k'<k,\alpha\in(0,1),a,b\in\R.$$ Then, we have for any $f\in C^{k+\alpha,\frac{k+\alpha}{4}}(\bar{\Omega}\times[a,b])$ that
	\begin{align}
	\|f\|_{C^{k'+\alpha,\frac{k'+\alpha}{4}}(\bar{\Omega}\times[a,b])}\le \|f\big|_{t=a}\|_{C^{k'+\alpha}(\bar{\Omega})}+C(b-a)^{\bar{\alpha}}\|f\|_{C^{k+\alpha,\frac{k+\alpha}{4}}(\bar{\Omega}\times[a,b])}.
	\end{align}
	Hereby, the constants $C$ and $\bar{\alpha}$ depend on $\alpha,k,k'$ and $\bar{\Omega}$. Especially, if $f\big|_{t=a}\equiv 0$, we have
	\begin{align}\label{EquationContractivityofLowerOrderTerms}
	\|f\|_{C^{k'+\alpha,\frac{k'+\alpha}{4}}(\bar{\Omega}\times[a,b])}\le C(b-a)^{\bar{\alpha}}\|f\|_{C^{k+\alpha,\frac{k+\alpha}{4}}(\bar{\Omega}\times[a,b])}.
	\end{align}
\end{lemma}
\begin{proof}
Cf. \cite[Lemma 2.17]{goesswein2019Dissertation}.
\end{proof}
As a final remark of this chapter we want to mention that we sometimes identify  Sobolev and H\"older spaces (in local coordinates) with Besov spaces, especially to use interpolation and composition results. As we do not need them on manifolds we will not introduce them here but they can be found, e.g., in \cite{Triebel1994TheoryofFunctionSpace}.

\subsection{Analytic Setting and Short Time Existence} \label{SubsectionShortTimeExistence}
Before moving on to the analysis of this paper we want to briefly summarize the analytic framework and the short time existence result we discussed in \cite{garckegoesweinpreprintshorttimeexistenceSDFTJ}. To get a suitable analytic setting we write the evolution of the triple junction clusters as normal graphs over a fixed reference triple junction cluster $\Gamma_{\ast}:=\Gamma^1_{\ast}\cup\Gamma^2_{\ast}\cup\Gamma^3_{\ast}\cup\Sigma_{\ast}$. Near the boundary we need to allow a tangential part in a neighborhood of the triple junction to describe general motions of the geometric object. So, we want to describe the evolving hypersurface $\Gamma$ as image of the diffeomorphism
\begin{align}
\Phi^{i}_{\boldsymbol{\rho},\boldsymbol{\mu}}: \Gamma^i_{\ast}\times [0,T]&\to \R^{n+1},\notag\\
(\sigma,t)&\mapsto \sigma+\rho^i(\sigma,  t)\Nast^i(\sigma)+\mu^i(\sigma,t)\tauast^i(\sigma), \label{EquationNormalGraphTJ}
\end{align}
where $\tau^i_{\ast}$ are fixed, smooth tangential vector fields on $\Gamma^i_{\ast}$ that equal $\nuast^i$ on $\Sigmaast$ and have a support in a neighborhood of $\Sigmaast$ in $\Gamma^i_{\ast}$. The tuple $(\boldsymbol{\rho},\boldsymbol{\mu})$ consists of the unknown functions for which we want to derive a PDE system. Hereby, $\bmu$ has to be given as function in the normal part as otherwise the resulting PDE problem will be degenerated. We know from the work of \cite{depner2014mean} that the condition
\begin{align}\label{EquationConditonForConcurrencyofTJ}
\Phi^1_{\boldsymbol{\rho},\boldsymbol{\mu}}(\sigma,t)=\Phi^2_{\boldsymbol{\rho},\boldsymbol{\mu}}(\sigma,t)=\Phi^3_{\boldsymbol{\rho},\boldsymbol{\mu}}(\sigma,t) & &\text{for }\sigma\in\Sigma_{\ast}, t\ge 0,
\end{align}
which guarantees concurrency of the triple junction, is equivalent to
\begin{align}\label{EquationEquivalentCondtionforTripleJunctionConservation}
\begin{cases}
\gamma^1\rho^1+\gamma^2\rho^2+\gamma^3\rho^3=0 &\text{on }\Sigma_{\ast},\\
\boldsymbol{\mu}=\mathcal{T}\boldsymbol{\rho} &\text{on }\Sigma_{\ast}.
\end{cases}
\end{align}
Hereby, the matrix $\mathcal{T}$ is given by
\begin{align*}
\mathcal{T}=\begin{pmatrix}0 & \frac{c^2}{s^1} & -\frac{c^3}{s^1} \\
-\frac{c^1}{s^2} & 0 & \frac{c^3}{s^2} \\ \frac{c^1}{s^3} & -\frac{c^2}{s^3} & 0 \end{pmatrix},
\end{align*}
with $s^i=\sin(\theta^i)$ and $c^i=\cos(\theta^i)$. The second line in (\ref{EquationEquivalentCondtionforTripleJunctionConservation}) implies that the tangential part $\boldsymbol{\mu}$ is uniquely determined on $\Sigmaast$ by the values of $\boldsymbol{\rho}$. This motivates to get rid of the degenerated degrees of freedom of $\boldsymbol{\mu}$ by setting 
\begin{align}
\mu^i(\sigma):=\mu(\pr^i_{\Sigma}(\sigma)),
\end{align} where $\pr^i_{\Sigma}$ denotes the projection from a point on $\Gamma^i_{\ast}$ to the nearest point on $\Sigma_{\ast}$. Note that this map is only well-defined on a neighborhood of $\Sigmaast$ in $\Gammaast^i$. But we only need $\mu^i$ to be defined on the support of $\tau^i_{\ast}$, which can be assumed to be arbitrary small. Therefore, this tangential part can be used to find a well-defined PDE formulation.
\begin{remark}[Analytic aspects of the non-local term]\ \\
Observe that $\bmu$ is non-local in space. For the existence analysis this is not a problem as it still induces only quasi-linear terms of lower order. Additionally, they depend continuously on $\brho$ and so they are easy to control. However, we will see in Section \ref{SectionTangentialPart} that they are not compatible with our \L ojasiewicz-Simon approach.
\end{remark}

To get a suitable PDE-setting for $\boldsymbol{\rho}$, we retract the equations from $\Gamma^i(t)$ on $\Gamma_{\ast}^i$. From here on, we will write $\Gamma_{\boldsymbol{\rho}}$ resp. $\Sigma_{\boldsymbol{\rho}}$ when referring to the triple junction cluster and the triple junction given as image of $\Phi_{\boldsymbol{\rho},\boldsymbol{\mu}}$. Also, we will use sub- and superscripts $i$ and $\boldsymbol{\rho}$ to denote pull-backs of quantities of the hypersurface $\Gamma_{\boldsymbol{\rho}}^i$ or of the triple junction $\Sigma_{\boldsymbol{\rho}}$.
This will also be applied on differential operators. So, for example, we will write for $(\sigma,t)\in\Gammaast^i\times[0,T], i=1,2,3,$
\begin{align*}
H^i_{\boldsymbol{\rho}}(\sigma,t)&:=H_{\Gamma^i_{\boldsymbol{\rho}}}(\Phi^i_{\boldsymbol{\rho},\boldsymbol{\mu}}(\sigma,t)),\\
\Delta_{\boldsymbol{\rho}}H_{\boldsymbol{\rho}}(\sigma,t)&:=\left(\Delta_{\Gamma^i_{\boldsymbol{\rho}}}H_{\Gamma^i_{\boldsymbol{\rho}}}\right)(\Phi^i_{\boldsymbol{\rho},\boldsymbol{\mu}}(\sigma,t)).
\end{align*}
Later, we will also sometimes use this notation when we consider the quantities on $\Gamma_{\boldsymbol{\rho}}$ but it will always be clear what is meant.\\
The evolution law \eqref{EquationMotionDuetoSurfaceDiffusion} together with the boundary conditions (CC), (AC), (CCP) and (FB) can now be written as the following problem on $\Gammaast$.
\begin{align}(SDFTJ)\label{EquationSurfaceDiffusionTripleJunctionGeometricVersiononRrenceFrame}
\begin{cases}
V^i_{\boldsymbol{\rho}}=-\Delta_{\brho}H^i_{\brho} &\text{on }\Gamma^i_{\ast}, t\in[0,T], i=1,2,3,\\
\gamma^1\rho^1+\gamma^2\rho^2+\gamma^3\rho^3=0 &\text{on }\Sigma_{\ast}, t\in[0,T],\\
\langle N^1_{\brho},N_{\brho}^2\rangle=\cos(\theta^3) &\text{on }\Sigma_{\ast}, t\in [0,T],\\
\langle N^2_{\brho}, N^3_{\brho}\rangle=\cos(\theta^1) &\text{on }\Sigma_{\ast}, t\in[0,T], \\
\gamma^1H^1_{\brho}+\gamma^2H^2_{\brho}+\gamma^3H^3_{\brho}=0 &\text{on }\Sigma_{\ast}, t\in[0,T],\\
\nabla_{\brho}H^1_{\brho}\cdot\nu_{\brho}^1=\nabla_{\brho}H^2_{\brho}\cdot\nu^2_{\brho} &\text{on }\Sigma_{\ast}, t\in[0,T], \\
\nabla_{\brho}H^2_{\brho}\cdot\nu^2_{\brho}=\nabla_{\brho}H^3_{\brho}\cdot\nu_{\brho}^3 &\text{on }\Sigma_{\ast}, t\in[0,T],\\
(\rho^i(\sigma,0),\mu^i(\sigma,0))=(\rho_0^i, \mu_0^i) &\text{on }\Gamma_{\ast}^i\times \Sigma_{\ast}, i=1,2,3.
\end{cases}
\end{align}
Here, we assume that the initial surfaces are given as $\Gamma^i_0=\Gamma^i_{\rho_0^i,\mu_0^i}, i=1,2,3$ for $\boldsymbol{\rho}_0$ small enough in the $C^{4+\alpha}$-norm and $\boldsymbol{\mu}_0=\mathcal{T}\boldsymbol{\rho}_0$. This will then guarantee that the $\Gamma_0^i$ are indeed embedded hypersurfaces, cf. \cite[Remark 1]{depner2014mean} .\\
We want to mention that one can rewrite $\eqref{EquationSurfaceDiffusionTripleJunctionGeometricVersiononRrenceFrame}_1$ as a parabolic equation
\begin{align}
\partial_t\rho^i=\mathcal{K}^i(\rho^i, \boldsymbol{\rho}|_{\Sigma_{\ast}}),
\end{align}
where the operator $$\mathcal{K}^i:C^{4+\alpha,1+\alpha/4}(\Gamma^i_{\ast}\times[0,T])\times C^{4+\alpha,1+\alpha/4}(\Sigma_{\ast}\times[0,T])\to C^{\alpha,\alpha/4}(\Gammaast\times[0,T]) $$ contains both derivatives of $\rho^i$ and $\boldsymbol{\rho}|_{\Sigma_{\ast}}$ of up to order four. A detailed calculation can be found in \cite[Page 38f.]{goesswein2019Dissertation}.\\
The boundary condition can be rewritten as operators $$\mathcal{G}^i: C^{4+\alpha,1+\alpha/4}_{TJ}(\Gammaast\times[0,T])\to C^{4-\sigma_i+\alpha,\frac{4-\sigma_i+\alpha}{4}}(\Sigmaast\times[0,T])$$ given by
\begin{align*}
\mathcal{G}^1(\boldsymbol{\rho})&:=\gamma^1\rho^1+\gamma^2\rho^2+\gamma^3\rho^3=0   &\text{ on }\Sigma_{\ast}\times[0,T],\\
\mathcal{G}^2(\boldsymbol{\rho})&:=\langle N^1_{\brho}, N^2_{\brho}\rangle-\cos(\theta^3)=0  &\text{ on }\Sigma_{\ast}\times[0,T],\\
\mathcal{G}^3(\boldsymbol{\rho})&:=\langle N^2_{\brho},N^3_{\brho}\rangle-\cos(\theta^1)=0  &\text{ on }\Sigma_{\ast}\times[0,T],\\
\mathcal{G}^4(\boldsymbol{\rho})&:=\gamma^1H^1_{\brho}+\gamma^2H^2_{\brho}+\gamma^3H^3_{\brho}=0   &\text{ on }\Sigma_{\ast}\times[0,T],\\
\mathcal{G}^5(\boldsymbol{\rho})&:=\nabla_{\brho}H^1_{\brho}\cdot\nu^1_{\brho}-\nabla_{\brho}H^2_{\brho}\cdot \nu^2_{\brho}=0    &\text{ on }\Sigma_{\ast}\times[0,T],\\
\mathcal{G}^6(\boldsymbol{\rho})&:=\nabla_{\brho}H^2_{\brho}\cdot\nu^2_{\brho}-\nabla_{\brho}H^3_{\brho}\cdot \nu^4_{\brho}=0    &\text{ on }\Sigma_{\ast}\times[0,T].
\end{align*}
With $\mathcal{G}(\boldsymbol{\rho}):=\big(\mathcal{G}^i(\boldsymbol{\rho})\big)_{i=1,...,6}$, \eqref{EquationSurfaceDiffusionTripleJunctionGeometricVersiononRrenceFrame} rewrites to the following  problem for $(\rho^1, \rho^2, \rho^3)$:
\begin{align}\label{EquationAnalyticFormulationofSDFTJ}
\begin{cases}
\partial_t\rho^i=\mathcal{K}^i(\rho^i, \boldsymbol{\rho}|_{\Sigma_{\ast}}) &\text{ on }\Gamma^i_{\ast}\times[0,T], i=1,2,3,\\
\mathcal{G}(\boldsymbol{\rho})=0 &\text{ on }\Sigma_{\ast}\times[0,T],\\
\rho^i(\cdot,0)=\rho^i_0 &\text{ on }\Sigma_{\ast}.
\end{cases}
\end{align}

As we seek for solutions smooth up to $t=0$, we will require some compatibility conditions for our initial data. Clearly, the boundary conditions given by $\mathcal{G}$ have to be fulfilled by the initial data. Additionally, for a smooth solution the conditions $\eqref{EquationSurfaceDiffusionTripleJunctionGeometricVersiononRrenceFrame}_1$ resp $\mathcal{G}^1(\brho)=0$ are differentiable in time at $t=0$ and so will also require a compatibility condition. Together, we get by considering \eqref{EquationSurfaceDiffusionTripleJunctionGeometricVersiononRrenceFrame} resp. \eqref{EquationAnalyticFormulationofSDFTJ} the corresponding compatibility conditions

\begin{align}
\label{EquationGeometricCompabilityConditionforSDFTJ}
(GCC)&\begin{cases}
\Gamma_0\text{ fulfils }(CC), (AC), (CCP), (FB) &\text{on } \Sigma_{0},\\
\sum_{i=1}^3\gamma^i\Delta_{\Gamma_0^i} H_{\Gamma_0^i}=0 &\text{on }\Sigma_0.
\end{cases}\\
\label{EquationCompabilityConditionsAnalyticVersionofSFDTJ}
(ACC)&\begin{cases}\mathcal{G}(\boldsymbol{\rho}_0)=0 & \text{on }\Sigma_{\ast}, \\ \mathcal{G}_0(\boldsymbol{\rho}_0):=\sum_{i=1}^3\gamma^i\mathcal{K}^i(\rho_0^i, \boldsymbol{\rho}_0\big|_{\Sigma_{\ast}})=0 & \text{on }\Sigma_{\ast}.
\end{cases} 
\end{align}
Indeed, one can prove that $(ACC)$ and $(GCC)$ are equivalent provided that $\brho_0$ is small enough in the $C^{4+\alpha}(\Gammaast)$-norm, cf. \cite[Lemma 4.1]{goesswein2019Dissertation}. In this framework we were able to prove in \cite{garckegoesweinpreprintshorttimeexistenceSDFTJ} the following short time existence result.
\begin{theorem}[Short time existence for surface diffusion flow with triple junctions] \label{TheoremSTETripleJunctions}\ \\
	Let $\Gammaast$ be a $C^{5+\alpha}$-reference cluster. Then there exists an $\varepsilon_0>0$ and a $T>0$ such that for all initial data $\brho_0\in C^{4+\alpha}_{TJ}(\Gammaast)$ with $\|\brho_0\|_{C^{4+\alpha}_{TJ}(\Gammaast)}<\varepsilon_0$, which fulfill the analytic compatibility conditions (\ref{EquationCompabilityConditionsAnalyticVersionofSFDTJ}), there exists a unique solution $\boldsymbol{\rho}\in C^{4+\alpha,1+\frac{\alpha}{4}}_{TJ}(\GammaastT)$ of (\ref{EquationAnalyticFormulationofSDFTJ}). Additionally, the norm of $\brho$ is bounded by a constant $R$ uniformly in $\brho_0$.
\end{theorem}
\section{Technical Problems Proving the \L ojasiewicz-Simon Gradient Inequality}\label{SectionTechnicalAspects}

In 1962 Stanis\l aw \L ojasiewicz found in his study of analytic functions an estimate for the distance of a point to the set of the roots of the function. In the same work \cite{lojasiewiczunepropertie} he realized  that his inequality can be used to show non-linear stability of equilibria of ordinary differential equations with a gradient flow structure. This idea was used by Simon in \cite{simonasymptoticsforaclassofnonlinear} to show the same results for PDEs with gradient flow structure and since then a lot of authors used it. The big advantage of this method is that one does not need to have a detailed knowledge of the set of equilibria and the stability argument is straight forward once a \L ojasiewicz-Simon gradient inequality, which we will abbreviate in the remaining work with LSI, is proven. A very important and general result on this is \cite{Chill2003OnTheLojasiewiczSimon}, wherein the author showed that the critical condition of analyticity of the first derivative of the energy functional (even $C^{\infty}$ is not enough to guarantee a LSI) needs only to hold on a subset called the critical manifold, which in most application is of finite dimension. Nevertheless, most authors use \cite[Corollary 3.11]{Chill2003OnTheLojasiewiczSimon}, for whose application most authors normally verify that the first derivative is analytic (on the whole space!) and the second derivative evaluated in the equilibrium is a Fredholm operator of index $0$. Hereby it is often suggested that the Hilbert structure is essential (see especially \cite{dallacquapozzispener2016lojasiewichsimongradientclampedcurves}) which strongly restricts the general result of \cite{Chill2003OnTheLojasiewiczSimon}. But actually, this is not the case and in \cite{FeehannMaridakisLojasiewiczSimon} an overview of different results about LSIs is given, from which we will use the following for our work, where we use the notation of \cite{Chill2003OnTheLojasiewiczSimon} to state the result.
\begin{proposition}[Feehan, Maridakis, 2015] \label{TheoremLSIFeehanMaridakis}\ \\
	Let $V,W$ be real Banach spaces with continuous embeddings $T': V\subset W$ and $T: W\hookrightarrow V'$ such that the embedding $j:=T\circ T':V\hookrightarrow V'$ is a definite embedding, that is, the bilinear form $V\times V\to\R, (x,y)\mapsto j(y)(x)$ is definite. \\
	Let $U\subset V$ be an open subset, $E: U\to\R$ a $C^2$-function such that $\mathcal{M}=E'$ is real analytic as a map with values in $W$. Furthermore, suppose that $x_{\infty}\in U$ is a critical point of $E$ and $\mathcal{L}=E''(x_{\infty})$ is a Fredholm operator with index zero as a map with values in $W$. Then, there are constants $C\in(0,\infty), \sigma\in (0,1]$ and $\theta\in(0,\frac{1}{2}]$ such that for all $x\in U$ with
	\begin{align*}
	\|x-x_{\infty}\|_V\le \sigma,
	\end{align*}
	we have
	\begin{align}\label{EquationLSIFeehanMaridakis}
	|E(x)-E(x_{\infty})|^{1-\theta}\le C\|\mathcal{M}(x)\|_{W}.
	\end{align}	
\end{proposition}
\begin{remark}[$\mathcal{M}$ and $\mathcal{L}$ as maps with values in $W$]\label{RemarkMandLasMappingswithvaluesinW}\ \\
	We want to note that if we consider $\mathcal{M}$ and $\mathcal{L}$ as maps with values in $W$ this is not completely precise. Formally, we are considering $T^{-1}\circ\mathcal{M}$ and $T^{-1}\circ\mathcal{L}$. When we introduce our setting later we will explain how the operators should be interpreted.
\end{remark}
So, the only prerequisite we need, to use the standard procedure, is definiteness of the embedding $V\rightharpoonup V'$ induced by the choice of the embedding $W\rightharpoonup V'$ in \cite[Hypothesis 3.4i)]{Chill2003OnTheLojasiewiczSimon}. In our situation - and we expect this to be true for any application, that is motivated by a gradient flow structure - the embedding is induced by a inner product and in this case definiteness of the embedding follows from definiteness of the product.\\
Now we want to discuss how to choose the spaces $V$ and $W$. For the application of the LSI to prove stability we need $W$ to be corresponding to the gradient flow setting. In our situation, we have an $H^{-1}$-gradient. Recall that for a Hilbert space $H$ and a linear map $F: H\to\R$, we have for any $x\in H$ that 
\begin{align*}
\nabla_H F(x)\in H,\quad F'(x)\in H'.
\end{align*}
From this we conclude that $\mathcal{M}(x)$ needs to be in the dual space of $H^{-1}$ that is by reflexivity of the Sobolev spaces just $H^1$. Therefore, we expect $W$ to be $H^1$. As we need $\mathcal{L}(0)$ to be a Fredholm operator the space $V$ is induced naturally by $W$ and the order of $\mathcal{L}(0)$. Up to lower order perturbation the second variation of the surface are is given by the Laplacian and so the expected space for $V$ is $H^3$. We want to note that there are also other possible choices. For example, one could also work with $V=H^1$ and $W=H^{-1}$. This is linked to the $\mathcal{H}^{-1}$-gradient flow formulation of the surface diffusion flow and so in some sense it is more natural. But then we get in higher space dimensions surfaces that are only varifolds which brings new difficulties.\\
So far we have not used the fact that \cite{Chill2003OnTheLojasiewiczSimon} is a general result in a Banach space setting. But as we just mentioned regularity of the surfaces is a more critical issue compared to the situation of the Willmore flow. There, the operator $\mathcal{L}$ is a fourth order differential operator and as the Willmore flow is a $L^2$-gradient flow this yields $V=H^4$ and thus $W=L^2$. Additionally, the works cited above restrict to curves and surfaces and for this space dimensions the involved Sobolev spaces will have Banach algebra structure and embed into $C^2$. As we want to work in arbitrary space dimensions we cannot guarantee this for $H^3$ any more and so we have to do a modification. Our first idea hereby was to choose the setting $V=W^{3,p}$ and $W=W^{1,p}$ for sufficiently large  $p$ such that the two properties mentioned hold for $V$. As we worked for our short time existence in a H\"older setting no additional regularity theory is needed. In the end this method has a problem in the application. Although we can prove a LSI in this setting yields an estimate by the $L^p$-norm of $\nabla_{\Gamma}H_{\Gamma}$. But we need an estimate by its $L^2$-norm as this is the quantity associated to the gradient flow structure of the surface diffusion flow\footnote{Recall that $\|\nabla_{\Gamma}H_{\Gamma}\|_{L^2}=\|-\Delta_{\Gamma}H_{\Gamma}\|_{\mathcal{H}^{-1}}$!}. This forces us to use interpolation arguments. Due to the H\"older-inequality we have for any $q>p$ that
\begin{align}
\|\nabla_{\Gamma}H_{\Gamma}\|_{L^p(\Gamma)}\le \|\nabla_{\Gamma}H_{\Gamma}\|_{L^2(\Gamma)}^{1-\bar{\theta}}\cdot \|\nabla_{\Gamma}H_{\Gamma}\|_{L^q(\Gamma)}^{\bar{\theta}},
\end{align}  
where the interpolation exponent $\bar{\theta}$ is given by
\begin{align}
\frac{1}{p}=\frac{1-\bar{\theta}}{2}+\frac{\bar{\theta}}{q}.
\end{align}
In principal, this would lead to the desired bound as we have bounds for $\|\nabla_{\Gamma}H_{\Gamma}\|_{L^q(\Gamma)}^{\bar{\theta}}$ due to parabolic smoothing for the surface diffusion flow. But we see that as $q\to\infty$ we get
\begin{align}
1-\bar{\theta}\to\frac{2}{p}.
\end{align}
That means that we have an upper bound for $1-\bar{\theta}$. But if we apply this interpolation procedure on the right-hand side of \eqref{EquationLSIFeehanMaridakis} this increases the \L ojasiewicz-Simon exponent by $\frac{1}{1-\bar{\theta}}$ and so by $\frac{p}{2}$. We then cannot guarantee that the LSI exponent is less than one. But we need this property to apply our stability argument and so this is not a possible solution. Using embedding theory for Besov spaces we were at least able to work with submanifolds in $\R^3$. But as we want to have results for any space dimension we have to find a new approach.\\
We then decided not to modify integrability but differentiability. The advantage in this approach is that on the scale of differentiability the interpolation exponent will get as good as we need it as long we can guarantee parabolic regularization and $C^k$-estimates for arbitrary large $k$, which again follow from parabolic smoothing properties of our flow.  \\
One might now argue that one does not need a Banach space setting after all but we want to note the only problem in our situation are the interpolation properties of Sobolev spaces. There might be other situations where this idea is useful.\\
These two problems are general problems when working with the surface diffusion flow in higher space dimensions. They also appear in the study of closed hypersurfaces, cf. \cite[Chapter 5]{goesswein2019Dissertation}. Now, we want to discuss additional problems induced by the existence of a boundary. Firstly, we will need to include some boundary conditions to guarantee the Fredholm property of the second derivative of the surface energy which is again the surface Laplacian. This conditions should be connected to our flow as for the stability argument itself the solution $\boldsymbol{\rho}$ of (SDFTJ) should be an admissible function. As we have a second order differential operator we can only allow for three conditions and the natural choice are the angle conditions and the concurrency of the triple junction. \\
The application of the \L ojasiewicz-Simon technique on the situation with non-linear boundary conditions was - according to our knowledge - not discussed, yet. This situation is much more complicated than linear boundary conditions, that, e.g. were discussed in \cite{abelswilke2007Convergencetoequilibiriumfor} or \cite{dallacquapozzispener2016lojasiewichsimongradientclampedcurves} as there one can just write these conditions into the space $V$. Our first idea to overcome this was to include these non-linear boundary conditions in $V$ by parametrizing them over the linearized boundary conditions. This is actually possible and we will use this during the proof of Lemma \ref{LemmaHigherTimeRegularity}. Also, we can include this in the parametrization of the non-linear volume constraints we will need in the situation of double-bubbles to. But then, we will get in the first derivative of the surface energy terms arising due to the parametrization of the non-linear boundary conditions. These terms cannot be fitted in our setting. So we have to find a different approach. The new idea is to write the boundary conditions in the energy $\widetilde{E}$ itself. The consequence of this is that the space $W$ needs to have some trace parts for this boundary conditions. At first glance, that is bad for our stability analysis later as there these trace parts do not appear in the energy. But as the solutions of (SDFTJ) fulfill these boundary conditions, these terms vanish in the stability analysis itself. 
\\
The second technical difficulty is the tangential part $\boldsymbol{\mu}$. In \eqref{EquationFirstDerivativeEtildetriplejunction} we see that in the first derivative of the surface energy at a point $\overline{\brho}\in V$ we get a term of the form
\begin{align}\label{EquationBadnassofnonlocalTangentialparts}
\sum_{i=1}^3\int_{\Gamma^i_{\ast}}\mu^i(\brho)g(\overline{\brho})\dH^n,
\end{align}
with some function $g$ in $\overline{\brho}$ and $\brho\in V$. To apply our method, we need \eqref{EquationBadnassofnonlocalTangentialparts} to be of the form
\begin{align}\label{EquationBadnessofnonlocaltparts2}
\sum_{i=1}^3\int_{\Gammaast^i}\rho^i\overline{g}(\overline{\brho})\dH^n,
\end{align}
with another function $\overline{g}$ in $\overline{\brho}$. Observe now that if $\brho$ vanishes on $\Sigmaast$ we will have $\mu^i(\brho)\equiv 0$ on $\Gamma$ and thus \eqref{EquationBadnassofnonlocalTangentialparts} vanishes as well. Then, applying the fundamental lemma of calculus of variations on \eqref{EquationBadnessofnonlocaltparts2} we see that 
\begin{align}
\overline{g}\equiv 0.
\end{align}
But as the functional corresponding to \eqref{EquationBadnassofnonlocalTangentialparts} is not the zero functional, we cannot write \eqref{EquationBadnassofnonlocalTangentialparts} in the formulation \eqref{EquationBadnessofnonlocaltparts2}. Thus, the non-local tangential part cannot be dealt with our LSI methods and so the only possible solution is to replace the tangential part by a local version that describes the same set of triple junction manifolds as long as $\boldsymbol{\rho}$ is small enough. 
\begin{remark}[Non-linear tangential parts]\ \\
	Although we did not study what happens in the case of non-linear tangential parts, we expect that these should in general not be a problem. One only needs to guarantee that their linearization fits in the chosen setting for $W\hookrightarrow V^{\ast}$ for the analyticity of the first derivative. The evaluation of the second derivative at the reference frame normally depends only on the evolution in normal direction and so the tangential part does not matter at all.
\end{remark}
\begin{remark}[The situation for the volume preserving mean curvature flow]\ \\
	Although we do not discuss it here we want to mention that our work basically shows also stability for the case of volume preserving mean curvature flow. One can use the same LSI and only has to see that one also has parabolic regularization for this flow which we expect to be true.
\end{remark}
\section{Parabolic Smoothing near Stationary Reference Frames}\label{SectionParabolicSmoothing}

An important tool in the estimates for the stability analysis will be that we have parabolic smoothing for the solution found in Theorem \ref{TheoremSTETripleJunctions}. Hereby, the usual idea is to use the found solution to write the coefficient functions as fixed functions. This yields then a linear problem on which one could apply again arguments similar to the linear analysis in \cite{garckegoesweinpreprintshorttimeexistenceSDFTJ}. Unfortunately, the fully non-linear angle conditions prevent us from doing this. As the coefficient functions are of the same order as the boundary itself they will not have enough regularity to derive higher regularity.\\
Therefore, we will use the parameter trick instead. We want to note that we only can prove parabolic smoothing in the case of reference frames that are stationary solutions. The reason for this is due to the fact that our linear analysis in \cite{garckegoesweinpreprintshorttimeexistenceSDFTJ} only considers linearization in the reference frame, not in a general solutions. In future, we hope to generalize this result to arbitrary reference frames. But as we consider in our stability analysis only standard double bubbles as reference frames this result will be enough for the work in this article.\\
The strategy of the proof splits into three steps. We will first use the parameter trick to show that away from $t=0$ the time derivative $\partial_t\brho$ inherits the space regularity of $\brho$. From this we will get that the space regularity of $\brho$ is increased by four orders. Finally, we can start a bootstrap procedure using the regularity we already have for $\brho$ from Theorem \ref{TheoremSTETripleJunctions}.
\begin{proposition}[Higher time regularity of solutions of (SDFTJ) near stationary double bubbles]\label{LemmaHigherTimeRegularity}  
	Let $k\in\N_{\ge 4}, \alpha\in (0,1), T>0, t_k\in (0,T]$. There are $\varepsilon_k,C_k>0$ with the following property. For any initial data $\brho_0\in C^{4+\alpha}_{TJ}(\Gammaast)$ with $\|\brho_0\|\le \varepsilon_k$ such that the solution $\brho$ of (\ref{EquationSurfaceDiffusionTripleJunctionGeometricVersiononRrenceFrame}) fulfills 
	\begin{align}\label{EquationRandomCondition}
	\brho\in C^{4+\alpha,1+\frac{\alpha}{4}}_{TJ}(\GammaastT)\cap C^{k+\alpha,0}_{TJ}(\Gammaast\times[t_k,T]),
	\end{align}
	we have the increased time regularity 
	\begin{align}
	\partial_t\brho \in C^{k+\alpha,0}_{TJ}(\Gammaast\times[t_k,T])
	\end{align}
	and we have the estimate
	\begin{align}\label{EquationHihgerRegulartiyforTimeDerivativeTripleJunctions}
	\|\partial_t\brho\|_{C^{k+\alpha,0}_{TJ}(\Gammaast\times[t_k,T])}\le \frac{C_k}{t_k}\|\brho_0\|_{C^{4+\alpha}_{TJ}(\Gammaast)}.
	\end{align}
\end{proposition}

\begin{remark}[Notation for higher regularity]\ \\
We want to note that we do some abuse of notation in \eqref{EquationRandomCondition}. Precisely, the space should be written as 
\begin{align*}
\left\{\brho\in C^{4+\alpha,1+\frac{\alpha}{4}}_{TJ}(\GammaastT) \big| \brho\big|_{[t_k,T]}\in C^{k+\alpha,0}_{TJ}(\Gammaast\times[t_k,T])\right\}.
\end{align*}
	In the following, every intersection of function spaces shall be understand in this way.
\end{remark}
\begin{proof}
	We first need to construct a parametrization of the non-linear boundary and compatibility conditions over the linear ones. For this we consider the spaces
	\begin{align}
	C^{4+\alpha,1+\frac{\alpha}{4}}_{TJ,LCC}(\GammaastT)&:=\left\{\brho\in C^{4+\alpha,1+\frac{\alpha}{4}}_{TJ}(\GammaastT)\Big|\ \mathcal{B}(\brho)\equiv 0,\ \mathcal{B}_0(\brho)\equiv 0\text{ on }\Sigmaast\right\},\\
	X_k&:=C^{4+\alpha,1+\frac{\alpha}{4}}_{TJ,LCC}(\GammaastT)\cap C^{k+\alpha,0}_{TJ}(\Gammaast\times[t_k,T]),\\
	Y^1&:=C^{\alpha}_{TJ}(\Sigmaast),\\
	Y^2_k&:=\Pi_{i=1}^6\left(C^{4+\alpha-\sigma_i,\frac{4+\alpha-\sigma_i}{4}}(\SigmaastT)\cap C^{k+\alpha,0}(\Sigmaast\times[t_k,T])\right).
	\end{align}
	Hereby, $\mathcal{B}$ denotes the linearized boundary operator, cf. \cite[Section 4]{garckegoesweinpreprintshorttimeexistenceSDFTJ}, and the sum condition in the first space is linked to the linearized compatibility conditions, cf. \cite[Equation (36)]{garckegoesweinpreprintshorttimeexistenceSDFTJ}. Additionally, the $\sigma_i$ denotes as before the order of the boundary conditions, that is,
	\begin{align}\label{EquationDefinitionofSigma}
	\sigma_1=0,\ \sigma_2=\sigma_3=1,\ \sigma_4=2,\ \sigma_5=\sigma_6=3.
	\end{align} 
	Now we consider on these spaces the operator 
	\begin{align*}
	\widetilde{G}: X_k\oplus Z_k&\to Y^1\times Y^2_k,\\
	(\boldsymbol{u}, \bar{\boldsymbol{u}})&\mapsto \left(\mathcal{G}_0((\boldsymbol{u}+\bar{\boldsymbol{u}})\big|_{t=0}), \mathcal{G}(\boldsymbol{u}+\bar{\boldsymbol{u}})\right).
	\end{align*}
	Hereby, $\mathcal{G}$ denotes the non-linear boundary operator and $\mathcal{G}_0$ is corresponds to the non-linear compatibility conditions, see \eqref{EquationCompabilityConditionsAnalyticVersionofSFDTJ}. $Z_k$ is a suitable complementary space of $X_k$ which we want to construct in the following. As we want to apply the inverse function theorem on $G$ we want to have that $\partial_2\widetilde{G}$ is a bijective mapping $Z_k\to Y^1\times Y^2_k$. Thus, the space $Z_k$ needs to contain exactly one representative for every possible value of the $\partial_2\widetilde{G}$, which is given by the linearized boundary and compatibility operator. Note that actually every complementary space $Z_k$ of $X_k$ will have this property. If $x,y\in Z_k$ fulfil $\partial_2\widetilde{G}(x)=\partial_2\widetilde{G}(y)$ then their difference is in $\ker(\partial_2\widetilde{G})$ and thus by construction in $X_k$. This implies now that $x-y=0$. \\
	For better readability we will write the construction of the space $Z_k$ in the following lemma.
	\begin{lemma}[Existence of a complementary space of $X_k$]\label{LemmaExistenceofComplementarySpaceXk}\ \\ There is a closed complementary space of $X_k$ in $C^{4+\alpha,1+\frac{\alpha}{4}}_{TJ}(\GammaastT)\cap C^{k+\alpha,0}_{TJ}(\Gammaast\times[t_k,T])$.
	\end{lemma}
	\begin{proof} 
		We will construct $Z_k$ as sum of two spaces, each fulfilling one of the two objectives. First, consider
		\begin{align*}
		\bar{Z}_1^k:=\left\{\mathfrak{b}\in \Pi_{i=1}^6C^{4+\alpha-\sigma_i,(4+\alpha-\sigma_i)/4}(\GammaastT):\mathfrak{b}\big|_{t=0}=0\right\}\cap\Pi_{i=1}^6C^{k+\alpha-\sigma_i,0}(\Sigmaast\times[t_k,T]).
		\end{align*}
		We can now apply the continuous solution operator from \cite[Theorem 5.1]{garckegoesweinpreprintshorttimeexistenceSDFTJ} with $\bar{Z}_1^k$ to get a subspace of $C^{4+\alpha,\frac{1+\alpha}{4}}_{TJ}(\GammaastT)$. Additionally, for this problem we can apply a standard localization argument on $[t_k,T]$ to get that these functions are also in $C^{k+\alpha,0}(\Gammaast\times[t_k,T])$ with corresponding energy estimates. Therefore, we get a continuous operator on $\bar{Z}_1^k$ with values in $X_k$ and thus its image, which we will call $Z_1^k$ is a closed subspace and therefore a Banach space. \\
		With the elements of $Z_1^k$ we can adjust the values of the linearized boundary operator away from $t=0$. Now we need another space to control the boundary values at $t=0$ and the compatibility condition. For this we consider the space 
		\begin{align*}
		\bar{Z}_2:=C^{\alpha}(\Sigmaast)\times \Pi_{i=1}^6 C^{4+\alpha-\sigma_i}(\Sigmaast).
		\end{align*}
		Now we construct for every $(\mathfrak{b}_0,\mathfrak{b})\in \bar{Z}_2$ the solution $\boldsymbol{u}_0$ of the elliptic system\footnote{As we can parametrize all $\Gamma^i$ over the same domain in $\R^n$ solvability of the system follows from \cite{agmondouglisNi1959estimates}. Note that the needed compatibility conditions were proven in \cite{garckegoesweinpreprintshorttimeexistenceSDFTJ}.}
		\begin{align*}
		-\sum_{i=1}^3\gamma^i\mathcal{A}^i_{all}\boldsymbol{u}_0&=\mathfrak{b}_0, & &\text{on }\Gammaast^i, i=1,2,3,\\
		\mathcal{B}(\boldsymbol{u}_0)&=\mathfrak{b}, & &\text{on }\Sigmaast.
		\end{align*}
		Hereby, we formally extend $\mathfrak{b}_0$ on $\Gammaast$. Now extending these functions constantly in time we get a set of functions in $C^{4+\alpha, 1+\frac{\alpha}{4}}_{TJ}(\GammaastT)$ which we call $Z_2$. Note that due to continuity of the solution operator of the elliptic problem the space $Z_2$ is closed. \\
		Now we can set $Z_k:=Z_1^k\times Z_2$. Obviously, we have $Z_1^k\cap Z_2=\{0\}$. Additionally, $Z$ is a closed subspace of $C^{4+\alpha, 1+\frac{\alpha}{4}}(\GammaastT)$ which we see by the following. Suppose that we have a convergent sequence \begin{align*}
		(z_n)_{n\in\N}=(z_n^1+z_n^2)_{n\in\N}\subset Z.
		\end{align*}
		Then $(z^1_n+z^2_n)\big|_{t=0}=z_n^2\big|_{t=0}$ converges in the $C^{4+\alpha}$-norm and as $Z_2$ is closed the (constantly in $t$ extended) limit $z^2$ is in $Z_2$. This implies that \begin{align*}
		(z_n^1)_{n\in\N}=(z_n^1+z_n^2)_{n\in\N}-(z_n^2)_{n\in\N} 
		\end{align*}
		converges also in the $C^{4+\alpha,1+\frac{\alpha}{4}}$-norm and as $Z_1$ is closed the limit $z^1$ is in $Z_1$. Hence, $Z$ is a closed subspace and thus a Banach space.
	\end{proof}
	We now continue the proof of Proposition \ref{LemmaHigherTimeRegularity}. Observe that $\partial_2\widetilde{G}(0,0)=(\mathfrak{B}_0,\mathfrak{B})$ and due to the construction of $Z$ this is now bijective. So, we can apply the implicit function theorem to get the existence of a unique function $\gamma$ defined on a neighborhood $U$ of $0$ with $\widetilde{G}(\boldsymbol{u},\gamma(\boldsymbol{u}))=0$.  The set 
	\begin{align} \label{EquationParametrizationBizzarr}
	\{\boldsymbol{u}+\gamma(\boldsymbol{u})|\boldsymbol{u}\in U\}
	\end{align}
	describes all functions fulfilling the non-linear boundary and compatibility conditions near $0$. In the following we will write $\bar{\gamma}:=\Id+\gamma$ and $P_{X_k}$ for the projection on $X_k$. For later use, we want to note that
	\begin{align}\label{EquationDerivategammainzero}
	\gamma'(0)\bu=0, \quad\forall \bu\in X_k.
	\end{align}
	To see this, we observe that due to $\tilde{G}(\gamma(\bu))=0$ in a neighborhood of $0$ we have that
	\begin{align}\label{EquationDeineMutter}
	0=\tilde{G}'(0)\bar{\gamma'(0)}\bu=\tilde{G}'(0)(\Id+\gamma'(0))\bu=\tilde{G}'(0)\bu+\tilde{G}'(0)\gamma'(0)\bu=\tilde{G}'(0)\gamma'(0)\bu.
	\end{align}
	In the last equality we used that $\bu$ is in the kernel of $\tilde{G}'(0)$. On the other hand, by construction of $\gamma$ we have that $\gamma'(0)\bu$ is in a complementary space of the kernel and so \eqref{EquationDeineMutter} implies directly \eqref{EquationDerivategammainzero}.
	\\ 
	Now, we are able to perform the main argument for the smoothing, i.e, the parameter trick argument. An example for this method can be found in \cite[Section 5.2]{pruesssimonett2016movinginterfacesandquasilinearparabolicevolutionequations}). In our situation, we have to replace $\brho$ by its parametrization $\bar{\gamma}(\bu)$ over the linearized boundary conditions .\\
	 For some small $0<\varepsilon<1$ we consider the map
	\begin{align*}
	G:(1-\varepsilon,1+\varepsilon)\times C^{4+\alpha}_{TJ,LCC}(\Gammaast)\times X_k&\to C^{4+\alpha}_{TJ,LCC}(\Gammaast)\times \widetilde{X}_k,\\
	(\lambda,\boldsymbol{u}_0, \boldsymbol{u})&\mapsto \left(\boldsymbol{u}\big|_{t=0}-\boldsymbol{u}_0,V_{\bar{\gamma}(\boldsymbol{u})}+\lambda\Delta_{\bar{\gamma}(\boldsymbol{u})}H_{\bar{\gamma}(\boldsymbol{u})}\right),
	\end{align*}
	where we use the spaces
	\begin{align*}
	C^{4+\alpha}_{TJ,LCC}(\Gammaast)&:=\left\{\bu\in C^{4+\alpha}_{TJ}(\Gammaast) | \mathcal{B}_0\bu\equiv 0, \mathcal{B}\bu\equiv 0\right\},\\
	\widetilde{X}_k&:=\left\{\boldsymbol{f}\in C^{\alpha,\frac{\alpha}{4}}_{TJ}(\GammaastT)\Bigg|\sum_{i=1}^3f^i\big|_{t=0}=0\text{ on }\SigmaastT\right\}\cap C^{k-4+\alpha,0}_{TJ}(\Gammaast\times[t_k,T]).
	\end{align*}
	Note that the second component $G_2$ indeed fulfils the compatibility condition of $\widetilde{X}_k$ as due to (\ref{EquationGeometricCompabilityConditionforSDFTJ}) we have that
	\begin{align*}
	\lambda\left(\Delta_{\overline{\gamma}(\bu)}^1H^1_{\overline{\gamma}(\bu)}+\Delta_{\overline{\gamma}(\bu)}^2H^2_{\overline{\gamma}(\bu)}+\Delta_{\overline{\gamma}(\bu)}^3H^3_{\overline{\gamma}(\bu)} \right)=0.
	\end{align*}
	The fact that the normal velocities sum up to zero was proven in the proof of \cite[Lemma 6.2]{garckegoesweinpreprintshorttimeexistenceSDFTJ}.  \\	
	Now, we want to check the prerequisites to apply the implicit function theorem. Observe that as $\Gammaast$ is a stationary solution of (SDFTJ) we have that $G(1,0,0)=0$. Furthermore, $G$ is an analytic operator as it can be written as the sum of products of linear, continuous maps and parabolic H\"older-spaces have a Banach algebra structure. Finally, we have for the partial Fr\'echet-derivative $\partial_{3}G(1,0,0)$ that
	\begin{align}
	\partial_3 G(1,0,0)\boldsymbol{u}=(\boldsymbol{u}|_{t=0}, (\partial_t-\mathcal{A}_{all})\boldsymbol{u}).
	\end{align} 
	Here, we used in the second component the chain rule and \eqref{EquationDerivategammainzero}.
	Due to the result\footnote{As mentioned in the proof of Lemma \ref{LemmaExistenceofComplementarySpaceXk} we use hereby, that for a linear problem we get higher regularity on $[t_k,T]$ using a standard localization argument.} from \cite[Theorem 5.1]{garckegoesweinpreprintshorttimeexistenceSDFTJ} we see that $\partial_{3}G(1,0,0)$ is bijective and then the implicit function theorem yields the existence of neighborhoods \begin{align}(1,0)\in U\subset (1-\varepsilon,1+\varepsilon)\times C^{4+\alpha}_{TJ,LCC}(\Gammaast)\end{align} and \begin{align}0\in V\subset C^{4+\alpha, 1+\frac{\alpha}{4}}_{TJ,LCC}(\GammaastT)\end{align} together with a unique, analytic function $\zeta: U\to V$ such that for all $(\lambda,\boldsymbol{u}_0)\in U$ we have that
	\begin{align}\label{EquationTimeSmoothnessHelpConstruction}
	G(\lambda, \boldsymbol{u}_0, \zeta(\lambda,\boldsymbol{u}_0))=0.
	\end{align}
	On the other hand, we may consider the solution of (SDFTJ) with initial data $\boldsymbol{\rho}_0$ denoted by $\boldsymbol{\rho}_{\boldsymbol{\rho}_0}$. Then the projection of the time-scaled function \begin{align}\boldsymbol{\rho}_{\boldsymbol{\rho}_0,\lambda}:=\boldsymbol{\rho_{\boldsymbol{\rho}_0}}(x,\lambda t)
	\end{align} on $X_k$ also solves (\ref{EquationTimeSmoothnessHelpConstruction}) with $\bu_0=\left(P_{X_k}(\boldsymbol{\rho}_{\boldsymbol{\rho}_0,\lambda})\big|_{t=0}\right)$ and by uniqueness we get \begin{align}\zeta(\lambda,\boldsymbol{u}_0)=P_{X_k}\left(\boldsymbol{\rho}_{\boldsymbol{\rho}_0,\lambda}\right).\end{align} 
	As we now can write $\boldsymbol{\rho}_{\boldsymbol{\rho}_0,\lambda}$ as composition of $\zeta$ and the smooth function $\gamma$, we derive smoothness of $\boldsymbol{\rho}_{\boldsymbol{\rho}_0,\lambda}$ in $\lambda$. Consequently, this implies
	\begin{align}\label{EquationRelationofSolutionsinParameterTrick}
	\partial_{\lambda}\boldsymbol{\rho}_{\boldsymbol{\rho}_0,\lambda}(t)=t\partial_t\boldsymbol{\rho}_{\boldsymbol{\rho}_0}(\cdot,\lambda t).
	\end{align}
	From this we conclude that $\partial_t\boldsymbol{\rho}_{\boldsymbol{\rho}_0}(\cdot,t)\in C^{k+\alpha}_{TJ}(\Gammaast)$ for all $t\in[t_k,T]$. Finally, we observe that 
	\begin{align}\label{Equation10102018}
	\|\partial_{\lambda}\gamma\left(\zeta(1,\boldsymbol{u}_0)\right)\|_{X_k}\le C\int_0^1\|\partial_2\partial_{\lambda}\zeta(1,s\boldsymbol{u}_0)\boldsymbol{u}_0\|_{X_k}ds\le \int_0^1 C\|\boldsymbol{u}_0\|_{C^{4+\alpha}_{TJ}(\Gammaast)}ds\le C\|\boldsymbol{\rho}_0\|_{C^{4+\alpha}}.
	\end{align}
	Here, we used in the first step analyticity of $\gamma$ - and consequently boundedness of its derivative in a neighborhood of $0$ -, in the second step the same argument for  $\zeta$ is analytic and in the last stept continuity of $P_{X_k}$. Combining \eqref{Equation10102018} with \eqref{EquationRelationofSolutionsinParameterTrick} we conclude \eqref{EquationHihgerRegulartiyforTimeDerivativeTripleJunctions}, which finishes the proof of Proposition \ref{LemmaHigherTimeRegularity}.
\end{proof} 
In the next step we want to use the gained time regularity to show additional regularity in space.
\begin{proposition}[Higher space regularity of solutions of (SDFTJ) near stationary dubble-bubbles]\label{LemmaHigherSpaceRegularityofSolutions} 
	For $T>0, k\in\N_{\ge 4}, t_k\in (0,T]$, there is $D_k, C_k'>0$ such that for all solutions \begin{align}\brho\in C^{4+\alpha,1+\frac{\alpha}{4}}_{TJ}(\GammaastT)\cap C^{k+\alpha,0}_{TJ}(\Gammaast\times[t_k,T])\end{align}
	with
	\begin{align}
	\partial_t\brho\in C^{k+\alpha,0}_{TJ}(\Gammaast\times[t_k,T]),\quad\|\partial_t\brho(t)\|_{C^{k+\alpha,0}_{TJ}(\Gammaast\times[t_k,T])}\le D_k,
	\end{align}
	we have that $\brho(t)\in C^{k+3+\alpha}_{TJ}(\Gammaast)$ for all $t\in[t_k,T]$ and 
	\begin{align}\label{EquationEstimateHohereRegularitatOrtTJ}
	\|\brho\|_{C^{k+3+\alpha,0}_{TJ}(\Gammaast\times[t_k,T])}\le \frac{C_k'}{t_k}\|\partial_t\brho\|_{C^{k+\alpha,0}_{TJ}(\Gammaast\times[t_k,T])}.
	\end{align}
\end{proposition} 
\begin{proof}
	We consider the operator
	\begin{align*}
	G: C^{k-1+\alpha}_{TJ}(\Gammaast)\times C^{k+3+\alpha}_{TJ}(\Gammaast)&\to C^{k-1+\alpha}_{TJ}(\Gammaast)\times\Pi_{i=1}^6 C^{k+3+\alpha-\sigma_i}(\Sigmaast),\\
	(\mathfrak{f}, \boldsymbol{\rho})&\mapsto (-\Delta_{\brho}H_{\brho}-\mathfrak{f}, \mathcal{G}(\boldsymbol{\rho})).
	\end{align*}
	Hereby, $\mathcal{G}$ denotes the non-linear boundary operator from (\ref{EquationAnalyticFormulationofSDFTJ}) and the $\sigma_i$ are as in \eqref{EquationDefinitionofSigma}. We have that $G(0,0)=0$ as $\Gammaast$ is a stationary solution of $(SDFTJ)$ and additionally we observe that due to the results from \cite[Section 4]{garckegoesweinpreprintshorttimeexistenceSDFTJ} we have
	\begin{align*}
	\partial_2 G(0,0)\boldsymbol{\rho}=(\mathcal{A}_{all}(\boldsymbol{\rho}),\mathcal{B}(\boldsymbol{\rho})).
	\end{align*}
	As we checked in \cite[Section 5.2]{garckegoesweinpreprintshorttimeexistenceSDFTJ}
	the Lopatinskii-Shapiro conditions for this system, we may apply the results from \cite{agmondouglisNi1959estimates} to see that $\partial_2 G(0,0)$ is bijective. Hence,  the implicit function theorem yields the existence of neighborhoods $0$ in $U\subset C^{k-1+\alpha}_{TJ}(\Gammaast)$ and  $0$ in $V\subset C^{k+3+\alpha}_{TJ}(\Gammaast)$ and a unique, smooth function $\zeta: U\to V$ fulfilling
	\begin{align}
	G(\mathfrak{f},\zeta(\mathfrak{f}))=0.
	\end{align}
	Now, we want to connect this with the solution $\brho$ of $(SDFTJ)$. For this we observe that for 
	\begin{align}
	\mathfrak{f}(\brho(t)):=V_{\brho}(t)=\left(\partial_t\brho(t)N_{\ast}+\partial_t\bmu(\brho)(t)\tau_{\ast}\right)\cdot N_{\brho(t)}
	\end{align} 
	we have $\mathfrak{f}(\brho(t))\in C^{k-1+\alpha}(\Gammaast)$ and
	\begin{align}
	G(\mathfrak{f}(\brho(t)),\brho(t))=0.
	\end{align}
	Due to uniqueness of $\zeta$ this shows $\zeta(\mathfrak{f}(\brho(t)))=\brho(t)$ and thus $\brho(t)\in C^{k+3+\alpha,0}_{TJ}(\Gammaast)$. The estimate (\ref{EquationEstimateHohereRegularitatOrtTJ}) can be proven as in (\ref{Equation10102018}).
\end{proof}
In the final step we start now a boot-strap procedure to get arbitrary high space regularity.
\begin{proposition}[$C^k$-regularity in space near stationary double bubbles]\label{PropositionHigherSpaceRegularitySFDTJ}\ \\
	Let $\brho$ be the solution of (SDFTJ) from Theorem \ref{TheoremSTETripleJunctions} with initial data $\brho_0$ and existence time $T$. For every $k\in\N_{\ge 4}$ and $t_k\in(0,T]$ there are $\varepsilon_k>0, C_k>0$ such that for all $\brho_0\in C^{4+\alpha}_{TJ}(\Gammaast)$ with $\|\brho_0\|\le \varepsilon_k$ we have
	\begin{align}\label{EquationEstimateHigherRegularitybrhoTJ}
	\brho\in C^{k+\alpha,0}_{TJ}(\Gammaast\times[t_k,T]), \quad \|\brho\|_{ C^{k+\alpha,0}_{TJ}(\Gammaast\times[t_k,T])}\le\frac{C_k}{t_k}\|\brho_0\|_{C^{4+\alpha}_{TJ}(\Gammaast)}.
	\end{align}
\end{proposition}
\begin{remark}[$C^{\infty}$-regularity]\ \\
	As we might have $\varepsilon_k\to 0$ for $k\to\infty$ this does not show $C^{\infty}$-regularity. But for our stability analysis it will be enough to choose $k$ big enough depending only on the dimension of the surrounding space.
\end{remark}
\begin{proof}
	Proposition \ref{LemmaHigherTimeRegularity} and Proposition \ref{LemmaHigherSpaceRegularityofSolutions} start a bootstrap procedure as for $k=4$ condition (\ref{EquationRandomCondition}) is already fulfilled due to Theorem \ref{TheoremSTETripleJunctions}. Then, in every step we gain three orders of differentiability in space, which shows the claim.
\end{proof}
\section{Choice of the Tangential Part}\label{SectionTangentialPart}
\setlength{\parindent}{0em} 
As we argued in Section \ref{SectionTechnicalAspects}, we need to get rid of the non-local tangential part. Nevertheless, we still want the tangential part to be given as a function in $\boldsymbol{\rho}$ as otherwise we would have different degrees of freedom for the first and second derivative of the surface area, which would be a problem in the proof of the LSI. The idea to achieve such a tangential part is surprisingly easy. We observe that the linear connection \eqref{EquationEquivalentCondtionforTripleJunctionConservation} between $\boldsymbol{\rho}$ and $\boldsymbol{\mu}$ is the same along the triple junction. This suggests to use the matrix $\mathcal{T}$ also in the interior to get at every point the tangential part $\bmu$ as function of the normal part $\brho$. The constructed function will then involve no projection on $\Sigma_{\ast}$ and thus be purely local. The only problem hereby is that for the calculation we have to evaluate all the $\rho_i$ in one point and technically each $\rho_i$ only exists on $\Gamma^i_{\ast}$. But we can solve this by identifying the three hypersurfaces via diffeomorphisms.\\
More precisely, for $B=B_1(0)\subset\R^{n}$ we choose $C^{\infty}$-parametrizations 
\begin{align}\label{EquationIdentificationoftheGammai}
\boldsymbol{\varphi}=(\varphi_1,\varphi_2,\varphi_3): B\to (\R^{n+1})^3
\end{align}
of $(\Gamma^1_{\ast},\Gamma^2_{\ast},\Gamma^3_{\ast})$ such that for all $x\in\partial B$ we have that 
\begin{align}
\varphi_1(x)=\varphi_2(x)=\varphi_3(x).
\end{align}	\label{PageWhyCanWeparametrizeoveronedomain}
We note here that it is indeed possible to parametrize each $\Gamma^i_{\ast}$ with one parametrization as $\Gamma^i_{\ast}$ is either a spherical cap or a flat ball in a hypersurface. Now, we can define the new tangential coefficient field $\bmu_G$ on $B$ via the linear connection from before and then do a pushfoward, that is,
\begin{align}\label{EquationNewTangentialPart}
\boldsymbol{\mu}_{G}(\boldsymbol{\rho}) :=(\mathcal{T}(\boldsymbol{\rho}\circ\boldsymbol{\varphi}^{-1}))\circ\boldsymbol{\varphi}^{-1}.
\end{align}
By construction this function fulfills the necessary condition for the concurrency of the triple junctions.
\begin{remark}[Linearization results for the new tangential part]\ \\
	In the next section we will use the same linearization results as in \cite[Section 4]{garckegoesweinpreprintshorttimeexistenceSDFTJ} for the boundary conditions with the original tangential part. This is indeed possible as  the calculations were done for a general tangential part and so they work both for the original tangential part $\eqref{EquationEquivalentCondtionforTripleJunctionConservation}_2$ and the new one $\eqref{EquationNewTangentialPart}$.
\end{remark}
Before we can use this new tangential term we have to check that by this procedure we describe the same triple junction manifolds as in Theorem \ref{TheoremSTETripleJunctions}, which is carried out in the following Lemma. There, we will denote by $\bmu_{DGK}$ the tangential part from Section \ref{SubsectionShortTimeExistence} and by $\bmu_G$ the tangential given by \eqref{EquationNewTangentialPart}.
\begin{lemma}[Equivalence of the tangential parts]\label{LemmaEuivalenceTangeitalPart}\ \\
	For every $k\ge 2$ we consider the space
	\begin{align}
	C^k_{TJ,0}(\Gammaast):=\{\brho\in C^k_{TJ}(\Gammaast)\big| \rho^i=0\text{ on }\Sigmaast, i=1,2,3\}.
	\end{align} 
	There exist $r,r'>0$ (depending on k) together with a map 
	\begin{align} 
	\bar{F}: C^k_{TJ}(\Gamma_{\ast})\supset B_r(0)\to B_{r'}(0)\subset C^k_{TJ,0}(\Gamma_{\ast}),
	\end{align}
	such that the map $F=\bar{F}+\Id$ fulfils for all $\brho\in B_r(0)$ that
	\begin{align}\label{LemmaEquivalencyofthetangentialparts}
	\Gamma_{\boldsymbol{\rho}, \mu_{DGK}(\boldsymbol{\rho})}=\Gamma_{F(\boldsymbol{\rho}),\mu_{G}(F(\boldsymbol{\rho}))}.
	\end{align}
\end{lemma}
\begin{proof}
	Consider for $X=C^k_{TJ}(\Gamma_{\ast}), Y=Z=C^k_{TJ,0}(\Gamma_{\ast})$ the map $G: X\times Y \to Z$
	\begin{align*}
	(\brho,\overline{\brho})\mapsto (x\mapsto d_H(x+(\rho^i+\overline{\rho}^i)(x)\nu^i_{\ast}(x)+(\boldsymbol{\mu_{G}}(\brho+\overline{\brho})(x))^i\tau^i_{\ast}(x),\Gamma_{\brho, \bmu_{DGK}(\brho)}^i))_{i=1,2,3}.
	\end{align*}
	Hereby, $d_H$ denotes the usual Hausdorff distance from a point to a compact set. Note that $G$ has indeed values in $Z$ as $\bmu_{DGK}$ and $\bmu_G$ equal on $\Sigmaast$. Therefore, $\Gamma_{\brho,\bmu_{DGK}(\brho)}$  and $\Gamma_{\brho+\overline{\brho},\bmu_G(\brho+\overline{\brho})}$ have the same triple junction as $\brho+\overline{\brho}=\brho$ on $\Sigmaast$. We want to use the implicit function theorem to find a map $\bar{F}$ with $G(\boldsymbol{\rho},\bar{F}(\boldsymbol{\rho}))=0$ for $\boldsymbol{\rho}$ small enough. Then, the surfaces $\Gamma^i_{F(\boldsymbol{\rho}),\mu_{G}(F(\boldsymbol{\rho}))}$ are subsets of $\Gamma^i_{\boldsymbol{\rho},\mu_{DGK}(\boldsymbol{\rho})}$. Additionally, they are simply connected and have the same boundary as $\Gamma^i_{\boldsymbol{\rho},\mu_{DGK}(\boldsymbol{\rho})}$ and consequently they have to be equal to $\Gamma^i_{\boldsymbol{\rho},\mu_{DGK}(\boldsymbol{\rho})}$. So, $\bar{F}$ fulfils (\ref{LemmaEquivalencyofthetangentialparts}). We claim that
	\begin{align}\label{EquationProofofReparametrizationofTangentialPart}
	\partial_2G(0,0)\overline{\boldsymbol{\rho}}=\overline{\brho}, \quad \overline{\boldsymbol{\rho}}\in C^2_{TJ}(\Gamma_{\ast}).
	\end{align}
	In order to see this we will calculate $G(0,\varepsilon \widetilde{\boldsymbol{\rho}})$ pointwise. There are two geometrical situations that could arise. The first is that $x+\overline{\rho}^i(x)\nu^i_{\ast}+(\mu_{G}(\overline{\boldsymbol{\rho}})(x))_i$ lies within the $R^i$-tube of $\Gamma^i_{\ast}$. The second possibility, which in theory could also arise, is that points near $\Sigma_{\ast}$ could leave the $R^i$ tube. We want to see that for $\overline{\boldsymbol{\rho}}$ small enough in $Y$ only the first situation is possible. For this we note that for any $\sigma\in \Gamma_{\ast}^i$ the ball with radius $\min(R^i,\|\sigma-\pr_{\Sigma_{\ast}}(\sigma)\|_{\R^{n+1}})$ is completely contained in the $R^i$-tube of $\Gamma_{\ast}^i$. Hereby, $\pr_{\Sigma_{\ast}}$ denotes the projection the projection on the nearest point on $\Sigmaast$. This is at least near $\Sigmaast$ well-defined. By elementary geometry we see that 
	\begin{align} 
	\|\sigma-\pr_{\Sigma_{\ast}}(\sigma)\|_{\R^{n+1}}=2R^i\sin\left(\frac{d_{\Gamma^i}(\sigma,\pr_{\Sigma_{\ast}}(\sigma))}{R^i}\right).
	\end{align}
	Thus, this quantity as function in $d_{\Gamma^i}(\sigma,\pr_{\Sigma_{\ast}}(\sigma))$ has a derivative larger than $1$ close to zero. On the other hand, we know that
	\begin{align*}
	\|x-(x+\overline{\rho}^i(x)\nu^i_{\ast}(x)+\mu_{G}(\overline{\rho}(x))^i\tau^i_{\ast}(x))\|_{\R^{n+1}}&\le C(|\overline{\rho}^i(x)|+\sum_{j=1}^3|\overline{\rho}^j(x)|)\\
	&\le Cd_{\Gamma^i_{\ast}}(x,\pr_{\Sigma_{\ast}}(x))\|\overline{\rho}\|_{C^1},
	\end{align*}   
	with the constant $C$ only depending on the linear relation between $\boldsymbol{\mu}$ and $\boldsymbol{\rho}$ but not on the point $x$. So, by choosing $r'$ small enough we can guarantee that the point $x+\overline{\rho}^i(x)\nu^i_{\ast}(x)+\mu_{G}(\overline{\rho}(x))_i\tau^i_{\ast}(x)$ lies within the ball within the ball around $x$ with radius $\min(R^i,\|x-\pr_{\Sigma_{\ast}}(x)\|_{\R^{n+1}})$ and so within the $R_{\ast}^i$-tube of $\Gamma_{\ast}^i$. Note that this argumentation is also true if $\Gammaast^i$ is a flat ball.\\
	We return now to the proof of \eqref{EquationProofofReparametrizationofTangentialPart} where we can now restrict to the first situation. If $\Gammaast^i$ is a flat ball this is clear as we then have for $\overline{\brho}\in Z$ and $\varepsilon>0$
	\begin{align}
	d_H(x+\varepsilon\overline{\rho}^i(x)\nu^i_{\ast}(x)+(\boldsymbol{\mu}_G(\varepsilon\overline{\brho})(x))^i\tau_{\ast}^i(x), \Gamma^i_{\ast})=d_H(x+\varepsilon\overline{\rho}^i(x)\nu^i_{\ast}(x), \Gamma^i_{\ast})=\varepsilon\overline{\rho}^i(x).
	\end{align} 
	Here, we used that in this case the tangential movement is parallel to $\Gamma^i_{\ast}$. Thus, this will not change the distance to $\Gammaast^i$. Taking the limit $\varepsilon\to 0$ we get easily \eqref{EquationProofofReparametrizationofTangentialPart}.\\
	We now consider the case that $\Gammaast^i$ is a spherical cap. From elementary geometry we know that the nearest point on $\Gamma^i_{\ast}$ from  $x+\overline{\rho}^i(x)\nu^i_{\ast}(x)+\mu_{G}(\overline{\rho}(x))_i\tau^i_{\ast}(x)$ is given by the intersection of the straight line between $x_{\overline{h},\mu_G(\overline{h})}$ and $M^i_{\ast}$ and $\Gamma^i_{\ast}$. By Pythagoras' theorem we get that
	\begin{align*}
	G\left(0,\varepsilon\overline{\boldsymbol{\rho}}\right)(x)&=\sqrt{\left(R^i_{\ast}+\varepsilon\overline{\rho}^i\right)^2+\|\mu_{G}^i(\varepsilon\overline{\boldsymbol{\rho}})(x)\tau_{\ast}^i(x)\|^2}-R^i_{\ast}\\
	&=(R^i_{\ast}+\varepsilon \overline{\rho}^i)+\frac{1}{2(R^i_{\ast}+\varepsilon\overline{\rho}_i)}\varepsilon^2\|\mu_{G}^i(\overline{\boldsymbol{\rho}})(x)\tau^i_{\ast}(x)\|^2+\sigma(\varepsilon^4)-R_i^{\ast}.
	\end{align*} 
	Here, we used $\mu_{G}^i(\varepsilon\overline{\rho})=\varepsilon\mu_{G}^i(\overline{\rho})$ and a Taylor expansion. This implies
	\begin{align*}
	\frac{G(0,\varepsilon\overline{\rho})-G(0,0)}{\varepsilon}=\overline{\rho}_i+\frac{1}{2(R_i^{\ast}+\varepsilon\overline{\rho_i})}\varepsilon\|\mu_{G}(\overline{\boldsymbol{\rho}})\tau^i_{\ast}(x)\|^2+\sigma(\varepsilon^4).
	\end{align*}
	This converges uniformly in $x$ to $\overline{\rho}^i(x)$ for $\varepsilon\to 0$ as $\mu_{G}^i(\overline{\boldsymbol{\rho}})(x)\tau^{\ast}(x)$ is bounded in $x$. This shows (\ref{EquationProofofReparametrizationofTangentialPart}). Continuity of $G$ and $\partial_2 G$ follows from the formulas in \cite[Section 2.2]{pruesssimonett2016movinginterfacesandquasilinearparabolicevolutionequations}, which we can apply as the points $x+(h_i+\overline{h}_i)(x)\nu^i_{\ast}(x)+(\boldsymbol{\mu_{G}}(h+\overline{h})(x))_i\tau^i_{\ast}(x)$ stay in the tubular neighborhoods of $\Gamma_{h, \mu_{DGK}(h)}^i$ for $r$ and $r'$ small enough. Therefore, we can apply the implicit function theorem to show the claim.
\end{proof}
In the following, when we write $\bmu$ we will always refer to $\bmu_{G}$ unless said otherwise and if we omit the tangential part in the notation, the used tangential part will always be $\bmu_{G}$.
\section{Parmetrization of the Set of Volume Preserving Double Bubbles}\label{SectionParametrization}
Next we need to rewrite the set of the volume preserving distance functions over a Banach space, i.e., the tangent space of this constraint. This condition is necessary as otherwise the needed Fredholm property of the second derivative will not be fulfilled. Our parametrization should look like the one in \eqref{EquationParametrizationBizzarr}. Thus, we will need a suitable complementary space of the set of functions fulfilling the linearized volume conditions. A natural choice would be the set of constant functions fulfilling the concurrency condition $\eqref{EquationEquivalentCondtionforTripleJunctionConservation}_1$. But it is more convenient to work with functions that vanish near $\Sigma_{\ast}$ to avoid additional tangential parts and so we will choose suitable bump functions. \\
To construct our setting we start with
\begin{align}
U:=\left\{\bu\in C^2_{TJ}(\Gammaast)\big|\gamma^1u^1+\gamma^2u^2+\gamma^3u^3=0\text{ on }\Sigma_{\ast} \right\}.     
\end{align}
Note that we can put the condition for concurrency of the triple junction already in $U$ as it is linear. Now, we have again to consider the tangent space of volume preserving evolutions, which leads to
\begin{align}\label{EquationMeanValueIntegralCondition}
U_1:=&\left\{\boldsymbol{v}\in U\Bigg| \int_{\Gamma^1_{\ast}}v^1d\mathcal{H}^n=\int_{\Gamma^2_{\ast}}v^2d\mathcal{H}^n=\int_{\Gamma^3_{\ast}}v^3d\mathcal{H}^n\right\}
\end{align}
The equality condition for the integrals follows directly from applying Reynolds' transport theorem on the evolution of the enclosed volume and corresponds to the mean value freeness condition we got in the previous chapter. \\
Now we need to construct a suitable complemented space of $U_1$. To this end, we choose $f^1\in C^{\infty}(\Gamma^1_{\ast})$ and $f^2\in C^{\infty}(\Gamma^2_{\ast})$ with
\begin{align*}
&\supp(f^1)\subset\subset \Gamma^1_{\ast}\backslash\supp(\tau_{\ast}^1), \int_{\Gamma^1_{\ast}}f^1d\mathcal{H}^n=1,\\
&\supp(f^2)\subset\subset\Gamma^2_{\ast}\backslash\supp(\tau_{\ast}^2), \int_{\Gamma^2_{\ast}}f^2d\mathcal{H}^{n}=1
\end{align*}
and set
\begin{align}
U_2:=\text{span}(f^1,f^2).
\end{align}
Here, we consider $f^1,f^2$ as functions in $C^{\infty}_{TJ}(\Gamma_{\ast})$ extending them by zero. The space $U_2$ is closed being finite dimensional. Observe that the choice of the support of $f^1$ and $f^2$ guarantees that addition of linear combinations of them to other distance functions $\boldsymbol{\rho}$ will not change the tangential part. Before we go on we have to check that $U_1$ and $U_2$ are indeed complementary spaces in $U$.
\begin{lemma}[Complementarity of $U_1$ and $U_2$]
	$U_1$ and $U_2$ are complementary spaces in $U$.
\end{lemma}
\begin{proof}
	For $a,b\in\R$ and $f=af^2+bf^3$ we have that
	\begin{align*}
	\int_{\Gammaast^1}f\dH^n=0, \int_{\Gammaast^2}f\dH^n=a, \int_{\Gammaast^3}f^3\dH^n=b.
	\end{align*}
	Thus, only for $a=b=0$ we have $f\in U_1$ and therefore $U_1\cap U_2=\{0\}$. On the other hand, for $\boldsymbol{\rho}\in U$ it holds that
	\begin{align*}
	&\boldsymbol{\rho}=\left(\boldsymbol{\rho}-\alpha f^2-\beta f^3\right)+\alpha f^2+\beta f^3,\\
	&\alpha:=\int_{\Gammaast^2}\rho^2\dH^n-\int_{\Gammaast^1}\rho^1\dH^n, \beta:=\int_{\Gammaast^3}\rho^3\dH^n-\int_{\Gammaast^1}\rho^1\dH^n.
	\end{align*}
	Hereby, we have that $(\boldsymbol{\rho}-\alpha f^2-\beta f^3)\in U_1$ and so we conclude $U=U_1+U_2$.
\end{proof}
Now we are able to prove the existence of the sought parametrization of the volume constraint.
\begin{lemma}[Parametrization of the volume constraint for triple junction manifolds]\label{LemmaParametrisierungVolumenbedingungTripleJunction}\ \\
	Let $U=U_1\oplus U_2$ be as above. Then, there exists an open neighborhood $\widetilde{U}_1$ of $0$ in $U_1$ and an open neighborhood $\widetilde{U}$ of $0\in U$ together with a unique map $\bar{\gamma}: \widetilde{U}_1\to U_2$ such that the map \begin{align}
	\gamma:=\Id+\bar{\gamma}: \widetilde{U}_1\to\widetilde{U}
	\end{align}
	parametrizes the subset of all functions $U$ that belong to triple junction manifolds fulfilling the volume constraints. Furthermore, $\bar{\gamma}$ is analytic and we have for the first derivative of $\bar{\gamma}$ for $\boldsymbol{v}_0\in \widetilde{U}_1,\boldsymbol{v}\in U_1$ that
	\begin{align}\label{EquationFirstDerivativeofgammatriplejunction}
	\bar{\gamma}'(\boldsymbol{v}_0)\boldsymbol{v}&=\begin{pmatrix}
	\frac{\int_{\Gamma^1_{\gamma(\boldsymbol{v}_0)}}\left(v^1N^1_{\ast}+\mu^1(\boldsymbol{v})\tau^1_{\ast}\right)\cdot N^1_{\gamma(\boldsymbol{v}_0)}\dH^n-\int_{\Gamma^2_{\gamma(\boldsymbol{v}_0)}}\left(v^2N^2_{\ast}+\mu^2(\boldsymbol{v})\tau_{\ast}^2\right)\cdot N^2_{\gamma(\boldsymbol{v_0})}\dH^n}{\int_{\Gamma^2_{\gamma(\boldsymbol{v}_0)}}N^2_{\ast}\cdot N^2_{\gamma(\boldsymbol{v}_0)}\dH^n}f^1 \\
	\frac{\int_{\Gamma^3_{\gamma(\boldsymbol{v}_0)}}\left(v^3N^3_{\ast}+\mu^3(\boldsymbol{v})\tau_{\ast}^3\right)\cdot N^3_{\gamma(\boldsymbol{v_0})}\dH^n-\int_{\Gamma^1_{\gamma(\boldsymbol{v}_0)}}\left(v^1N^1_{\ast}+\mu^1(\boldsymbol{v})\tau^1_{\ast}\right)\cdot N^1_{\gamma(\boldsymbol{v}_0)}\dH^n}{\int_{\Gamma^3_{\gamma(\boldsymbol{v}_0)}}N^3_{\ast}\cdot N^3_{\gamma(\boldsymbol{v}_0)}\dH^n}f^2
	\end{pmatrix}.
	\end{align}
\end{lemma}
\begin{proof}
	Consider the functional 
	\begin{align*}
	G: U=U_1\oplus U_2&\to \R^2,\\
	(\boldsymbol{v},\boldsymbol{w})=\boldsymbol{u}&\mapsto \begin{pmatrix}\Vol(\Omega_{13}^{\boldsymbol{u}})-V_{1}^{\ast} \\ \Vol(\Omega_{23}^{\boldsymbol{u}})-V_2^{\ast} \end{pmatrix}.
	\end{align*}
	We want to apply the implicit function theorem and so we have to consider the partial derivative $\partial_2 G(0,0)$. For $\boldsymbol{w}=af^1+bf^2, a,b\in\R$ we get by transport theorems
	\begin{align*}
	\partial_2G(0,0)\boldsymbol{w}=\begin{pmatrix}\int_{\Gamma^1_{\ast}}af^1\dH^n \\ -\int_{\Gamma^2_{\ast}}bf^2\dH^n
	\end{pmatrix}=\begin{pmatrix}
	a \\ -b
	\end{pmatrix}. 
	\end{align*}
	This shows that $\partial_2 G(0,0): W\to \R^2$ is bijective. Continuity of $\partial_2 G$ and analyticity of $G$ follow by similar arguments as in the proof of Lemma \ref{LemmaAnalyticityofwideEtildeTriplejunction} and so we get the existence of an analytic function $\bar{\gamma}$ with the desired properties.\\
	Now, for $\bar{\gamma}'$ we do an implicit differentiation and rearrange the terms afterwards to get (\ref{EquationFirstDerivativeofgammatriplejunction}). We hereby used the fact that $\bar{\gamma}'$ can be seen as element in $\R^2$ and that $\bar{\gamma}(\boldsymbol{v})$ does not change the tangential part of $\boldsymbol{v}$. This yields then \eqref{EquationFirstDerivativeofgammatriplejunction}.
\end{proof}
\begin{remark}
	One could also calculate $\bar{\gamma}''(0)$ using the same procedure as in the case of closed hypersurfaces. But for the application later we only need $\widetilde{E}''(0)$ and for its calculation we can also use the results from \cite{HutchingsMorganProofDoubleBubbleConjecture}. 
\end{remark}
\section{Variational Formulas}\label{SectionVariationFormula}

Next we want to derive the formulas for the first and second derivative of the surface energy with the volume constraint. Before doing so we want to specify the notation for the different energies arising in the setting. By the plain $E$ we denote the surface energy as a functional on the set of triple junction manifolds near the considered stationary point $\Gammaast$. The functional $\widehat{E}$ arises from $E$ by the parametrization of these triple junction manifolds using distance functions $\boldsymbol{\rho}$ and the associated tangential part $\boldsymbol{\mu}_{G}(\boldsymbol{\rho})$. Finally, we restrict $\widehat{E}$ on the set of distance functions belonging to triple junction manifolds fulfilling the volume constraints, which we parametrize using the function $\gamma$ constructed in the previous section. In total, we get the energy functional
\begin{align}
\widetilde{E}: C^2_{TJ,C,(0)}(\Gammaast)\supset \widetilde{V}\to\R, \boldsymbol{\rho}\mapsto \sum_{i=1}^3 \Area\left(\Gamma^i_{\gamma(\boldsymbol{\rho})}\right).
\end{align}
on the space
\begin{align}
C^2_{TJ,C,(0)}(\Gammaast)=\left\{\boldsymbol{\rho}\in C^2_{TJ}(\Gammaast)\big|\sum_{i=1}^3\rho^i=0\text{ on }\Sigmaast,\int_{\Gammaast^1}\rho^1\dH^n=\int_{\Gammaast^2}\rho^2\dH^n=\int_{\Gammaast^3}\rho^3\dH^n\right\}.
\end{align}
The variation of this can again be calculated using the surface transport theorem and for $\boldsymbol{\rho}_0\in \widetilde{V}, \boldsymbol{\rho}\in C^2_{TJ,C,(0)}(\Gammaast)$ this leads to
\begin{align}\label{EquationFirstDerivativeEtildetriplejunction}
\widetilde{E}'(\boldsymbol{\rho}_0)\boldsymbol{\rho}=&-\sum_{i=1}^3\int_{\Gamma^i_{\gamma(\boldsymbol{\rho}_0)}}H^i_{\gamma(\boldsymbol{\rho}_0)}\left(\rho^iN^i_{\ast}+\mu^i(\boldsymbol{\rho})\tau^i_{\ast}+\left[\gamma '(\boldsymbol{\rho}_0)\boldsymbol{\rho}\right]N^i_{\ast}\right)\cdot N^i_{\gamma(\boldsymbol{\rho}_0)}\dH^n\\
&-\int_{\Sigma_{\gamma(\boldsymbol{\rho}_0)}}\sum_{i=1}^3\left(\rho^iN^i_{\ast}+\mu^i(\boldsymbol{\rho})\tau_{\ast}^i \right)\cdot\nu^i_{\gamma(\boldsymbol{\rho}_0)}\dH^{n-1}.\notag
\end{align}
Hereby, we used that $\gamma'(\boldsymbol{\rho}_0)\boldsymbol{\rho}$ vanishes on $\Sigma_{\ast}$. For later use we need to rewrite the boundary term in (\ref{EquationFirstDerivativeEtildetriplejunction}) such that it does not depend on $\rho^1$. Using $\rho^1=-\rho^2-\rho^3$ and 
\begin{align*}
\mu^1&=-\frac{1}{\sqrt{3}}(\rho^2-\rho^3),\\ \mu^2&=-\frac{1}{\sqrt{3}}(\rho^3-\rho^1)=-\frac{1}{\sqrt{3}}(2\rho^3+\rho^2),\\ \mu^3&=-\frac{1}{\sqrt{3}}(\rho^1-\rho^2)=\frac{1}{\sqrt{3}}(2\rho^2+\rho^3),
\end{align*}
we can rewrite it as
\begin{align}
-\int_{\Sigmaast}&\rho^2(N^2_{\ast}\cdot\nu^2_{\grhonull}-N^1_{\ast}\cdot\nu^1_{\grhonull}-\frac{1}{\sqrt{3}}\tau_{\ast}^1\cdot\nu^1_{\grhonull}-\frac{1}{\sqrt{3}}\tau_{\ast}^2\cdot\nu^2_{\grhonull}\notag\\
&+\frac{2}{\sqrt{3}}\tau_{\ast}^3\cdot\nu^3_{\grhonull})J_{\Sigma_{\grhonull}}
+\rho^3(N_{\ast}^3\cdot\nu^3_{\grhonull}-N^1_{\ast}\cdot\nu^1_{\grhonull}+\frac{1}{\sqrt{3}}\tau_{\ast}^1\cdot\nu^1_{\grhonull}\label{EquationBoundaryTermininEstrichtriplejunction}\\
&-\frac{2}{\sqrt{3}}\tau_{\ast}^2\cdot\nu^2_{\grhonull}+\frac{1}{\sqrt{3}}\tau_{\ast}^3\cdot\nu^3_{\grhonull})J_{\Sigma_{\grhonull}}\dH^{n-1}\notag
\end{align}
We also need to rewrite the tangential and the $\bar{\gamma}$-part in the first line in (\ref{EquationFirstDerivativeEtildetriplejunction}). For the first one we have
\begin{align}
&-\sum_{i=1}^3\int_{\Gamma_{\grhonull}^i}H^i_{\grhonull}\mu^i(\boldsymbol{\rho})\tau_{\ast}^i\cdot N^i_{\grhonull}\dH^n=\int_{\Gamma^1_{\grhonull}}H^1_{\grhonull}\frac{1}{\sqrt{3}}(\rho^2-\rho^3)\tau_{\ast}^1\cdot N_{\grhonull}^1\dH^n\notag\\
&+\int_{\Gamma^2_{\grhonull}}H^2_{\grhonull}\frac{1}{\sqrt{3}}(\rho^3-\rho^1)\tau_{\ast}^2N^2_{\grhonull}\dH^n+\int_{\Gamma^3_{\grhonull}}H^3_{\grhonull}\frac{1}{\sqrt{3}}(\rho^1-\rho^2)\tau_{\ast}^3\cdot N^3_{\grhonull}\dH^n\label{EquationReformulationoftangentialpartinEtildestrichtriplejunction}\\
&=\sum_{i=1}^3\int_{\Gamma^i_{\ast}}\rho^i\frac{1}{\sqrt{3}}(H^{i-1}_{\grhonull}\tau_{\ast}^{i-1}\cdot N^{i-1}_{\grhonull}-H^{i+1}_{\grhonull}\tau_{\ast}^{i+1}\cdot N^{i+1}_{\grhonull})J_{\Gamma^i_{\grhonull}}\dH^n. \notag
\end{align}
Hereby, an $i-1=-1$ has to be read as $3$. Additionally, we did abuse of notation by just considering $H^{i-1}, \tau^{i-1}_{\ast}, N^{i-1}$ resp. $H^{i+1}, \tau^{i+1}_{\ast}, N^{i+1}$ as functions on $\Gamma^i_{\ast}$. To be precise one would need to include a pullback. To fit the $\bar{\gamma}'$-term in our framework for the proof of the LSI we actually need to rewrite it. For the closed situation an explicit calculation of this can be found in \cite[Lemma 5.7]{goesswein2019Dissertation}. The same procedure works in this situation. But we will skip this here for the sake of readability.\\
The second variation of $\widetilde{E}$ was already calculated in \cite[Proposition 3.3]{HutchingsMorganProofDoubleBubbleConjecture} varying twice in the same direction. For the variation in two different directions $\boldsymbol{\rho},\bar{\boldsymbol{\rho}}\in C^2_{TJ,C,(0)}(\Gammaast)$ we then get by polarization that
\begin{align}\label{EquationEtildestrichstrichtriplejunction}
\widetilde{E}''(0)\boldsymbol{\rho}\bar{\boldsymbol{\rho}}=-\sum_{i=1}^3\int_{\Gamma^i_{\ast}}\left(\Dast\rho^i+|II^i_{\ast}|^2\rho^i\right)\bar{\rho}^i\dH^n-\int_{\Sigma_{\ast}}\sum_{i=1}^3 (\partial_{\nu^i_{\ast}}\rho^i-q^i\rho^i)\bar{\rho}^i\dH^{n-1}.
\end{align}
with $q^1=(\kappa^1_{\ast}+\kappa^3_{\ast})/\sqrt{3}, q^2=(\kappa^1_{\ast}-\kappa^3_{\ast})/\sqrt{3}$ and $q^3=(-\kappa^1_{\ast}-\kappa^2_{\ast})/\sqrt{3}$. Recall that the movement induced by $\boldsymbol{\mu}$ is purely tangential at the reference surface and therefore vanishes. Again with the next section in mind we rewrite the boundary term using $\bar{\rho}^1=-\bar{\rho}^2-\bar{\rho}^3$ as
\begin{align}\label{EquationBoundaryPartEtildestrichstrichtriplejunction}
\int_{\Sigma_{\ast}}(\partialnuasta{2}\rho^2-\partialnuasta{1}\rho^1-q^2\rho^2+q^2\rho^1)\bar{\rho}^2+(\partialnuasta{3}\rho^3-\partialnuasta{1}\rho^1-q^3\rho^3+q^1\rho^1)\bar{\rho}^3\dH^{n-1}.
\end{align}
\section{Prerequisites for the \L ojasiewicz-Simon Gradient Inequality} \label{SectionProofofLsi}

Now we can proceed to the proof of the LSI in the case of triple junction manifolds. We first need a setting for \cite{FeehannMaridakisLojasiewiczSimon} and due to the consideration from Section \ref{SectionTechnicalAspects} we set for $m>2+\frac{n}{2}, m\in\N$
\begin{align}
V&=\left\{\boldsymbol{\rho}\in H^{m}_{TJ}(\Gammaast)\Bigg|\sum_{i=1}^3\rho^i=0\text{ on }\Sigmaast,\int_{\Gamma^1_{\ast}}\rho^1d\mathcal{H}^n=\int_{\Gamma^2_{\ast}}\rho^2d\mathcal{H}^n=\int_{\Gamma^3_{\ast}}\rho^3d\mathcal{H}^n\right\},\label{EquationSettingLSITriplejunctionSpaceV}\\
W&=H^{m-2}_{TJ}(\Gammaast)\times \left(H^{m-\frac{3}{2}}(\Sigmaast)\right)^2,\\
W&\hookrightarrow V^{\ast}: (\mathfrak{f}, \mathfrak{b}^2,\mathfrak{b}^3)\mapsto\left(\boldsymbol{\rho}\mapsto \sum_{i=1}^3 \int_{\Gamma^i_{\ast}}\mathfrak{f}^i\rho^i\dH^n+\int_{\Sigma_{\ast}}\mathfrak{b}^2\rho^2+\mathfrak{b}^3\rho^3\dH^{n-1}\right).\label{EquationSettingforLSITripleJunction}
\end{align}
Here, we get only two copies of $H^{m-\frac{3}{2}}_{TJ}(\Sigmaast)$ as one degree of freedom is lost due to the sum condition in $V$. Note that due to our choice of $m$ we have that  $V\hookrightarrow C^2$. In particular, $V$ and $W$ are Banach algebras. Reminding Remark \ref{RemarkMandLasMappingswithvaluesinW} we note that due to \eqref{EquationFirstDerivativeEtildetriplejunction}, \eqref{EquationBoundaryTermininEstrichtriplejunction} and \eqref{EquationReformulationoftangentialpartinEtildestrichtriplejunction} this means we consider $\widetilde{E}'$ as map with values in $W$ in the following way:
\begin{align}
\widetilde{E}'(\brho_0)&=\left(\mathfrak{f}(\brho_0),\mathfrak{b}^2(\brho_0),\mathfrak{b}^3(\brho_0)\right), \quad \brho_0\in V, \label{EquationEtildestrichalsElementinWTJ}  \\
\mathfrak{f}^i(\brho_0)&=\left(H^i_{\gamma(\boldsymbol{\rho}_0)}\left(N^i_{\ast}+\bar{\gamma}'(\boldsymbol{\rho}_0)^iN^i_{\ast}\right)\cdot N^i_{\gamma(\boldsymbol{\rho}_0)}+\frac{1}{\sqrt{3}}\sum_{j=1}^2(-1)^j(H^{i+j}_{\grhonull}\tau_{\ast}^{i+j}\cdot N^{i+j}_{\grhonull})\right)J_{\Gamma^i_{\grhonull}},\notag\\
\mathfrak{b}^2(\brho_0)&=N^2_{\ast}\cdot\nu^2_{\grhonull}-N^1_{\ast}\cdot\nu^1_{\grhonull}-\frac{1}{\sqrt{3}}\tau_{\ast}^1\cdot\nu^1_{\grhonull}-\frac{1}{\sqrt{3}}\tau_{\ast}^2\cdot\nu^2_{\grhonull}+\frac{2}{\sqrt{3}}\tau_{\ast}^3\cdot\nu^3_{\grhonull}J_{\Sigma_{\grhonull}},\notag \\
\mathfrak{b}^3(\brho_0)&=N_{\ast}^3\cdot\nu^3_{\grhonull}-N^1_{\ast}\cdot\nu^1_{\grhonull}+\frac{1}{\sqrt{3}}\tau_{\ast}^1\cdot\nu^1_{\grhonull}-\frac{2}{\sqrt{3}}\tau_{\ast}^2\cdot\nu^2_{\grhonull}+\frac{1}{\sqrt{3}}\tau_{\ast}^3\cdot\nu^3_{\grhonull}J_{\Sigma_{\grhonull}}.\notag
\end{align}
Using \eqref{EquationEtildestrichstrichtriplejunction} and \eqref{EquationBoundaryPartEtildestrichstrichtriplejunction}, we get for $\widetilde{E}''(0)$ as mapping $V\to W$ that
\begin{align}
\boldsymbol{\rho}\mapsto \begin{pmatrix}
-\Delta_{\Gamma_{\ast}}\boldsymbol{\rho}-|II_{\ast}|^2\boldsymbol{\rho} \\  \partialnuasta{2}\rho^2-\partialnuasta{1}\rho^1-q^2\rho^2+q^1\rho^1 \\
\partialnuasta{3}\rho^3-\partialnuasta{1}\rho^1-q^3\rho^3+q^1\rho^1
\end{pmatrix}, \quad \brho\in V. \label{EquationwidetildeEstrichstrichnullalsElementinWTJ}
\end{align}
Now, we want to verify the prerequisites to apply Theorem \ref{TheoremLSIFeehanMaridakis} to
$\widetilde{E}: V\to \R$. We begin with the analyticity of the first variation.
\begin{lemma}[Analyticity of $\widetilde{E}'$]\label{LemmaAnalyticityofwideEtildeTriplejunction}\ \\
	There are neighborhoods 
	\begin{align*}
	0\in U\subset V, 0\in \widetilde{U}\subset \left\{\boldsymbol{\rho}\in H^m_{TJ}(\Gammaast)\Bigg|\sum_{i=1}^3\rho^i=0\text{ on }\Sigma_{\ast}\right\}
	\end{align*}
	such that the following maps are analytic
	\begin{enumerate}[i.)]
		\item $\widetilde{U}\to H^{m-1}_{TJ}(\Gammaast, \R^{n+1}), \boldsymbol{\rho}\mapsto \partial_j^{\boldsymbol{\rho}}$ for all $j=1,...,n$,\\
		$\widetilde{U}\to \left(H^{m-\frac{3}{2}}(\Sigmaast,\R^{n+1})\right)^3, \boldsymbol{\rho}\mapsto \partial_j^{\boldsymbol{\rho}}$ for all $j=1,...,n-1$.
		\item $\widetilde{U}\to H^{m-1}_{TJ}(\Gammaast), \boldsymbol{\rho}\mapsto g_{jk}^{\boldsymbol{\rho}}, \boldsymbol{\rho}\mapsto g_{\boldsymbol{\rho}}^{jk}$ for all pairs $j,k\in\{1,...,n\}$,\\
		$\widetilde{U}\to \left(H^{m-\frac{3}{2}}(\Sigmaast)\right)^3, \boldsymbol{\rho}\mapsto g_{jk}^{\boldsymbol{\rho}}, \boldsymbol{\rho}\mapsto g_{\boldsymbol{\rho}}^{jk}$ for all pairs $j,k\in\{1,...,n-1\}$.
		\item $\widetilde{U}\to H^{m-1}_{TJ}(\Gammaast), \boldsymbol{\rho}\mapsto J_{\Gamma_{\boldsymbol{\rho}}}$,\\
		$\widetilde{U}\to \left(H^{m-\frac{3}{2}}(\Gammaast)\right)^3, \boldsymbol{\rho}\mapsto J_{\Sigma_{\boldsymbol{\rho}}}$.
		\item $\widetilde{U}\to H^{m-1}_{TJ}(\Gammaast, \R^{n+1}), \boldsymbol{\rho}\mapsto N_{\boldsymbol{\rho}}$,\\
		$\widetilde{U}\to H^{m-1}_{TJ}(\Gammaast), \boldsymbol{\rho}\mapsto N_{\boldsymbol{\rho}}\cdot N_{\ast}\\
		\widetilde{U}\to H^{m-1}_{TJ}(\Gammaast),\boldsymbol{\rho}\mapsto N_{\boldsymbol{\rho}}\cdot\tau_{\ast}$.
		\item $\widetilde{U}\to \left( H^{m-\frac{3}{2}}(\Sigmaast, \R^{n+1})\right)^3, \boldsymbol{\rho}\mapsto \nu_{\boldsymbol{\rho}}$,\\
		$\widetilde{U}\to \left(H^{m-\frac{3}{2}}(\Sigmaast)\right)^3, \boldsymbol{\rho}\mapsto \nu_{\boldsymbol{\rho}}\cdot N_{\ast},\boldsymbol{\rho}\mapsto\nu_{\boldsymbol{\rho}}\cdot \tau_{\ast}$. 
		\item $\widetilde{U}\to H^{m-2}_{TJ}(\Gammaast), \boldsymbol{\rho}\mapsto h_{jk}^{\boldsymbol{\rho}}$ for all pairs $j,k\in\{1,...,n\}$.
		\item $\widetilde{U}\to H^{m-2}_{TJ}(\Gammaast), \boldsymbol{\rho}\mapsto H_{\boldsymbol{\rho}}$.
		\item $U\to H^m_{TJ}(\Gammaast), \boldsymbol{\rho} \mapsto \gamma(\boldsymbol{\rho})$.
		\item $U\to W, \boldsymbol{\rho}\mapsto \widetilde{E}'(\boldsymbol{\rho})$.
	\end{enumerate} 
\end{lemma}
\begin{remark}[Argumentation in local coordinates]\ \\
	Remember that we parametrized all hypersurfaces over the same domain, cf. Section \ref{PageWhyCanWeparametrizeoveronedomain}. Therefore, the $\partial_i, g_{ij}, g^{ij}$ and $h_{ij}$ may be considered as global quantities.
\end{remark}
\begin{proof}
	Before we begin with the proof itself we want to remind that pullback and pushforward of function spaces on $\Gammaast$ to function spaces in local coordinates are analytic operators. So, we can work in the Sobolev spaces of local coordinates.\\
	For the transformation of the $\partial_i$ we calculated in \cite[Section 4]{garckegoesweinpreprintshorttimeexistenceSDFTJ} that
	\begin{align*}
	\partial_j^{\boldsymbol{\rho}}=\partial_j^{\ast}+\left(\partial_j\boldsymbol{\rho}\right)N_{\ast}+\boldsymbol{\rho}\partial_jN_{\ast}+\partial_j\boldsymbol{\mu}(\boldsymbol{\rho})\tau_{\ast}+\boldsymbol{\mu}(\boldsymbol{\rho})\partial_j\tau_{\ast}.
	\end{align*}
	The first summand is constant, the second and third clearly linear in $\boldsymbol{\rho}$. As $\boldsymbol{\mu}$ is also linear in $\rho$ and partial derivatives are linear operators this shows that the other terms are linear in $\boldsymbol{\rho}$. Thus, $\partial_j^{\boldsymbol{\rho}}$ is affin-linear in $\boldsymbol{\rho}$ and therefore analytic. As due to our choice of $m$ we have that $H^{m-1}_{TJ}(\Gammaast)$ is a Banach algebra, the first part of ii.) follows directly from i.) as products of analytic functions between Banach algebras are again analytic. For the $g^{jk}$ we use the fact that the inverse matrix is an analytic operator. For iii.) we use that $J_{\Gamma_{\boldsymbol{\rho}}}$ resp. $J_{\Sigma_{\rho}}$ are given by $\sqrt{g_{\Gamma_{\boldsymbol{\rho}}}}$ resp. $\sqrt{g_{\Sigma_{\boldsymbol{\rho}}}}$ and these quantities are analytic due to theory for composition operators and the fact that the determinant and the square root are analytic on suitable domains. \\
	Using the multi-linear structure of the crossproduct and the analyticity of $\partial_j^{\boldsymbol{\rho}}$ we conclude analyticity of the crossproduct of the $\partial_j^{\boldsymbol{\rho}}$. Its normalization is then analytic due to composition operator theory and so we get the first part of $iv.)$ from which we get the results for the other functions. With analyticity of $N_{\boldsymbol{\rho}}$ we can argue for $v.)$ in the same way as the outer conormal is given as normalized crossproduct of $N_{\boldsymbol{\rho}}$ and the $\partial_j^{\boldsymbol{\rho}}$ for $j=1,...,n-1$.  For the $h_{jk}^{\boldsymbol{\rho}}$ we have to study the second derivatives for which we get
	\begin{align*}
	\partial_j\partial_k\varphi_{\boldsymbol{\rho}}=h_{jk}^{\ast}&+(\partial_j\partial_k\boldsymbol{\rho})N_{\ast}+\partial_k\boldsymbol{\rho}\partial_j N_{\ast}+\partial_j\boldsymbol{\rho}\partial_kN_{\ast}+\boldsymbol{\rho}\partial_j\partial_k N_{\ast}\\
	&+\partial_j\partial_k\boldsymbol{\mu}(\boldsymbol{\rho})\tau_{\ast}+\partial_k\boldsymbol{\mu}(\boldsymbol{\rho})\partial_j\tau_{\ast}+\partial_j\boldsymbol{\mu}(\boldsymbol{\rho})\partial_k\tau_{\ast}+\boldsymbol{\mu}(\boldsymbol{\rho})\partial_j\partial_k\tau_{\ast}.
	\end{align*}
	Again using the fact that $\boldsymbol{\mu}$ is linear in $\boldsymbol{\rho}$ we see that this is an affin-linear function in $\boldsymbol{\rho}$ and thus analytic. Then, $h_{jk}^{\boldsymbol{\rho}}$ is analytic as product of analytic functions. Analyticity of $H_{\boldsymbol{\rho}}$ is a consequence of ii.) and vi.). For viii.) we note that with our newly derive knowledge we get that the function $G$ from Lemma \ref{LemmaParametrisierungVolumenbedingungTripleJunction} is analytic and thus by the analytic version of the implicit function theorem (cf. Corollary \ref{AnalyticVersionOfImplicitFuntion}) $\gamma$ is.\\
	Now we remind what $\widetilde{E}'$ as function with values in $W$ actually is, cf. \eqref{EquationEtildestrichalsElementinWTJ}. Analyticity of these expressions follows now from the results i.)-viii.), which finishes the proof.
\end{proof}  
Now it remains to show that $\widetilde{E}''(0)$ is a Fredholm operator of index 0.
\begin{lemma}[Fredholm property of $\widetilde{E}''(0)$]\label{LemmaFredholmpropertyofwideEtildeTriplejunction}\ \\
	The map $\widetilde{E}''(0): V\to W$ is a Fredholm operator of index $0$.	
\end{lemma}
\begin{proof}
	We remind here that \eqref{EquationwidetildeEstrichstrichnullalsElementinWTJ} gives us $\widetilde{E}''(0)$ as map with values in $W$. Our goal is to prove bijectivity of the main part. Then, the claim follows as compact perturbations preserve the Fredholm index.\\
	Therefore, we consider the elliptic problem
	\begin{align}
	-\Dasta{i}\rho^i&=f^i & &\text{on }\Gasti, i=1,2,3,\label{EquationWhatever}\\
	\rho^1+\rho^2+\rho^3&=0 & &\text{on }\Sigmaast,\\
	\partialnuasta{2}\rho^2-\partialnuasta{1}\rho^1&=b^2 & &\text{on }\Sigmaast,\\
	\partialnuasta{3}\rho^3-\partialnuasta{1}\rho^1&=b^3 & &\text{on }\Sigmaast,
	\end{align}
	for $\boldsymbol{f}\in L^2_{TJ}(\Gammaast), b^2,b^3\in L^2(\Sigma_{\ast})$. We observe that for a classical solution $\boldsymbol{\rho}$ and a test function $\boldsymbol{\psi}\in C^2_{TJ,C}(\Gammaast)$ multiplication of $\ref{EquationWhatever}^i$ with $\psi^i$, integrating over $\Gamma^i_{\ast}$ and summing over $i$ yields to
	\begin{align*}
	\sum_{i=1}^3\int_{\Gammaast^i}f^i\psi^i\dH^n&=\sum_{i=1}^3\int_{\Gammaast^i}(-\Dast\rho^i)\psi^i\dH^n=\sum_{i=1}^3\intasti\gast\rho^i\cdot\gast\psi^i\dH^n-\intsigmaast\sum_{i=1}^3\left(\gast\rho^i\cdot\nuast^i\right)\psi^i\dH^{n-1}\\
	&=\sum_{i=1}^3\intasti\gast\rho^i\cdot\gast\psi^i\dH^n-\intsigmaast\partialnuasta{1}\rho^1(-\psi^2-\psi^3)+\partialnuasta{2}\rho^2\psi^2+\partialnuasta{3}\rho^3\psi^3\dH^{n-1}\\
	&=\sum_{i=1}^3\intasti\gast\rho^i\cdot\gast\psi^i\dH^n-\intsigmaast\left(\partialnuasta{2}\rho^2-\partialnuasta{1}\rho^1\right)\psi^2+\left(\partialnuasta{3}\rho^3-\partialnuasta{1}\rho^1\right)\psi^3\dH^{n-1}\\
	&=\sum_{i=1}^3\intasti\gast\rho^i\cdot\gast\psi^i\dH^n-\intsigmaast b^2\psi^2+b^3\psi^3\dH^{n-1}.
	\end{align*}
	Therefore, we get the following weak formulation. We set
	\begin{align*}
	\mathcal{E}&:=\left\{\ \boldsymbol{\rho}\in H^1_{TJ}(\Gammaast)\big|\gamma^1\rho^1+\gamma^2\rho^2+\gamma^3\rho^3=0\text{ on }\Sigmaast\right\},\\
	B(\boldsymbol{\rho},\boldsymbol{\psi})&:=\sum_{i=1}^3\int_{\Gamma^i_{\ast}}\gasta{i}\rho^i\cdot\gast\psi^i\dH^n,\quad \forall \boldsymbol{\rho},\boldsymbol{\psi}\in\mathcal{E},\\
	F(\boldsymbol{\psi})&:=\sum_{i=1}^3\int_{\Gamma^i_{\ast}}f^i\psi^i\dH^n+\int_{\Sigmaast}b^2\psi^2+b^3\psi^3\dH^{n-1}, \quad \forall \boldsymbol{\psi}\in \mathcal{E},
	\end{align*}
	and consider the problem
	\begin{align}\label{EquationHastudueinProblem}
	B(\boldsymbol{\rho},\boldsymbol{\psi})&=F(\boldsymbol{\psi}),\quad \forall\boldsymbol{\psi}\in\mathcal{E}.
	\end{align}
	Due to Lemma \ref{LemmaPoincareineuqlaityfortriplejunction} we get that $B$ is a continuous, coercive bilinear form and then the Lax-Milgram lemma gives the existence of a unique solution $\boldsymbol{\rho}\in \mathcal{E}$ to \eqref{EquationHastudueinProblem}. For $\boldsymbol{f}\in H^{m-2}_{TJ}(\Gammaast)$ and $b^2,b^3\in H^{m-\frac{3}{2}}(\Sigmaast)$ we can apply locally elliptic regularity theory from \cite{agmondouglisNi1959estimates} to get that $\boldsymbol{\rho}$ is actually in $H^m_{TJ}(\Gammaast)$. Note that the necessary complementary conditions were proven in \cite[Lemma 3]{depner2014mean}. From this we conclude that the operator
	\begin{align*}
	\boldsymbol{\rho}\mapsto \begin{pmatrix}
	-\Delta_{\Gamma_{\ast}}\boldsymbol{\rho} \\  \partialnuasta{2}\rho^2-\partialnuasta{1}\rho^1\\
	\partialnuasta{3}\rho^3-\partialnuasta{1}\rho^1
	\end{pmatrix},
	\end{align*}
	is bijective as a map from  $V$ to $W$, which finishes the proof.
\end{proof}
With this we deduce the desired LSI for the surface energy. 
\begin{theorem}[LSI for the parametrized surface area of triple junctions]\label{TheoremLSITripleJunctions} \ \\
	The \L ojasiewicz-Simon gradient inequality for $\widetilde{E}$ holds in $0\in V$ for the setting (\ref{EquationSettingLSITriplejunctionSpaceV})-(\ref{EquationSettingforLSITripleJunction}) for $V$ and $W$, i.e,. there exists $\sigma, C>0$ and $\bar{\theta}\in(0, \frac{1}{2}]$ such that for all $x\in V$ with $\|x\|\le\sigma$ it holds that
	\begin{align}\label{EquationLSITripleJunctionAnalytic}
	|\widetilde{E}(x)-\widetilde{E}(0)|^{1-\bar{\theta}}\le C ||\widetilde{E}'(x)\|_{W}.
	\end{align}
\end{theorem}
\begin{proof}
	The embedding $V\subset W\hookrightarrow V'$ is definite as the $L^2$-product is. The point $0$ is indeed a critical point of $\widetilde{E}$ as we have $\bar{\gamma}'(0)=0$ due to (\ref{EquationFirstDerivativeofgammatriplejunction}) and then because of (\ref{EquationFirstDerivativeEtildetriplejunction}) for any $\boldsymbol{\rho}\in V$ that
	\begin{align*}
	\widetilde{E}(0)\boldsymbol{\rho}&=-\sum_{i=1}^3\int_{\Gamma_{\ast}^i}H^i_{\ast}\rho^i\dH^n-\int_{\Gamma^i_{\ast}}\sum_{i=1}^3\mu^i(\rho^i)\dH^{n-1}\\
	&=-\int_{\Gamma^1_{\ast}}(-H^2_{\ast}-H^3_{\ast})\rho^1\dH^n-\int_{\Gamma^2_{\ast}}H^2_{\ast}\rho^2\dH^n-\int_{\Gamma^3_{\ast}}H^3_{\ast}\rho^3\dH^n=0,
	\end{align*}
	where we used in the second equality that both the $\mu^i$ and the $H^i_{\ast}$ add up to zero and in the third equality the constraint of $\boldsymbol{\rho}$. The claims follows now from Theorem \ref{TheoremLSIFeehanMaridakis} using Lemma \ref{LemmaAnalyticityofwideEtildeTriplejunction} and \ref{LemmaFredholmpropertyofwideEtildeTriplejunction}.
\end{proof}
\section{Stability Analysis of Standard Double Bubbles} \label{SectionStabilityAnalysis}
With our result from the last section we are now able to perform our stability analysis. To state our result precisely, we will prove the following:
\begin{theorem}[Stability of stationary double bubbles wit respect to the surface diffusion flow]\label{TheoremStabilityofStationarydoublebubbles}\ \\
	Let $\alpha>0$. Every stationary double bubble $\Gammaast$ is stable with respect to the surface diffusion flow in the following sense: \vspace{-0.2cm}
	\begin{itemize}
		\item[i.)] The stationary solution $\brho\equiv 0$ of $(SDFTJ)$ is Lyapunov stable with respect to the $C^{4+\alpha}$-norm. \vspace{-0.2cm}
		\item[ii.)] There is an $\varepsilon_S>0$ such that for all $\brho_0\in C^{4+\alpha}_{TJ}(\Gammaast)$ fulfilling $\|\brho\|\le \varepsilon_S$ and the compatibility conditions \eqref{EquationCompabilityConditionsAnalyticVersionofSFDTJ} the solution from Theorem \ref{TheoremSTETripleJunctions} converges for $t\to\infty$ to some $\brho_{\infty}$. Furthermore, $\Gamma_{\brho_{\infty}}$ is a (possibly different) standard double bubble.
	\end{itemize}
\end{theorem}
Before we can prove this we need to transform Theorem \ref{TheoremLSITripleJunctions} into a more geometric formulation. This comes in handy as we will need later the gradient flow structure of the surface diffusion flow, which is more natural on a geometric level.
\begin{lemma}[Geometric LSI for the Surface Energy on Triple Junctions]\label{LemmaGeometricLSITJ}\ \\
	Consider for $k>\max(2+\frac{n}{2},5)$ and $k':=\frac{2-\bar{\theta}}{\bar{\theta}}k+2$ with $\bar{\theta}$ from Theorem \ref{TheoremLSITripleJunctions} and any $R>0$ the set
	\begin{align*}
	Z_R:=B_R(0)\subset H^{k'}_{TJ}(\Gammaast).
	\end{align*} 
	Then, there is a $\sigma>0$ and a $C(T)>0$ only depending on $R$ such that for all
	\begin{align}
	\brho\in \left\{H^k_{TJ}(\Gammaast)\cap Z_R\big|\rho^1+\rho^2+\rho^3=0\text{ on }\Sigmaast\right\},
	\end{align}
	such that $\Gamma_{\brho}$ fulfils the volume constraints
	\begin{align}
	\Vol(\Omega_{12}^{\brho})=V^2_{\ast}, \quad \Vol(\Omega_{13}^{\brho})=V^2_{\ast},
	\end{align}
	and the angle conditions $\eqref{EquationSurfaceDiffusionTripleJunctionGeometricVersiononRrenceFrame}_3, \eqref{EquationSurfaceDiffusionTripleJunctionGeometricVersiononRrenceFrame}_4$ with $\theta^i=\frac{2\pi}{3}, i=1,2,3,$, we have that
	\begin{align}\label{EquationHilfsaufssageGeometriceLSITripleJunction}
	\|\widetilde{E}'(\widetilde{\boldsymbol{\rho}})\|_W\le C(X_2)\sum_{i=1}^3 \|\widetilde{E}'(\widetilde{\boldsymbol{\rho}})\|_{L^2(\Gammaast^i)}^{\left(\frac{2-2\bar{\theta}}{2-\bar{\theta}}\right)},
	\end{align} 
	where $\widetilde{\boldsymbol{\rho}}$ is the projection of $\brho$ on $U_1$ induced by the map $\gamma$ from Lemma \ref{LemmaParametrisierungVolumenbedingungTripleJunction}. In particular, we get the following geometric LSI: there is a (possible smaller) $\sigma>0$, $C>0$ such that we have for $\theta:=\frac{\bar{\theta}}{2}$ and all $\brho\in X_1\cap X_2$ that
	\begin{align}
	|E(\Gamma_{\boldsymbol{\rho}})-E(\Gammaast)|^{1-\theta}&\le C(X_2)\sum_{i=1}^3\left(\int_{\Gamma^i_{\boldsymbol{\rho}}}|\nabla_{\boldsymbol{\rho}}H^i_{\boldsymbol{\rho}}|^2\dH^n\right)^{1/2}\notag\\
	&\le C(X_2) \sqrt{ \sum_{i=1}^3\int_{\Gamma^i_{\boldsymbol{\rho}}}|\nabla_{\Gamma^i_{\boldsymbol{\rho}}}H_{\Gamma^i_{\boldsymbol{\rho}}}|^{^{^2}}\dH^n}.\label{EquationGeometricLSITripleJunction}
	\end{align}
\end{lemma}
\begin{proof}
	We first observe that the angle conditions guarantee that the $H^{k-\frac{3}{2}}(\Sigmaast)^2$-part in the $W$-norm of $\widetilde{E}'(\boldsymbol{\rho})$ vanishes. On each of the remaining three terms we use that $W$ is an interpolation space between $L^2$ and $H^{k'}$ together with our uniform bounds for the $H^{k'}$-norm from our parabolic smoothing result Proposition \ref{PropositionHigherSpaceRegularitySFDTJ} to get (\ref{EquationHilfsaufssageGeometriceLSITripleJunction}). We apply this now on \eqref{EquationLSITripleJunctionAnalytic} to get
	\begin{align}
	|E(\Gamma_{\boldsymbol{\rho}})-E(\Gammaast)|^{1-\theta}\le C(R)\left(\sum_{i=1}^3\|\widetilde{E}'(\widetilde{\boldsymbol{\rho}})\|_{L^2(\Gammaast^i)}^{\left(\frac{2-2\bar{\theta}}{2-\bar{\theta}}\right)} \right)^{\frac{2-\bar{\theta}}{2-2\bar{\theta}}}\le C(R)\sum_{i=1}^3 \|\widetilde{E}'(\widetilde{\boldsymbol{\rho}})\|_{L^2(\Gammaast^i)},
	\end{align}
	where we used Jensen's inequality in the last step. Now it remains to study the $L^2$-norm on the right-hand-side. As we have a Poincar\'e inequality on the space $V$ from (\ref{EquationSettingLSITriplejunctionSpaceV}) with $m=1$, we can control these terms by
	\begin{align}
	C(X_2)\sum_{i=1}^3\int_{\Gammaast^i}|\nabla_{\boldsymbol{\rho}}\mathfrak{f}^i|^2\dH^n,
	\end{align}
	with
	\begin{align}
	\mathfrak{f}^i:=\left(H^i_{\gamma(\boldsymbol{\rho}_0)}\left(N^i_{\ast}+\bar{\gamma}'(\boldsymbol{\rho}_0)^iN^i_{\ast}\right)\cdot N^i_{\gamma(\boldsymbol{\rho}_0)}+\frac{1}{\sqrt{3}}\sum_{j=1}^2(-1)^j(H^{i+j}_{\grhonull}\tau_{\ast}^{i+j}\cdot N^{i+j}_{\grhonull})\right)J_{\Gamma^i_{\grhonull}}.
	\end{align} 
	This can be taken care of similarly to the argument in \cite[Lemma 5.11, (5.31)f.]{goesswein2019Dissertation} and leads to the first line in (\ref{EquationGeometricLSITripleJunction}). For the second line we apply equivalence of norms on the vector $$\boldsymbol{x}:=\left(\sqrt{\int_{\Gamma^i(t)}|\nabla_{\Gamma^i(t)}H_{\Gamma^i(t)}|^2\dH^n}\right)_{i=1,2,3}$$ in $\R^3$.
\end{proof}
As before we can now use this geometric version to prove stability. In the following, $k, k', \sigma$ are chosen as in Lemma \ref{LemmaGeometricLSITJ} and $\varepsilon_0$ and $T$ are the bounds for the initial data and the existence time from Theorem \ref{TheoremSTETripleJunctions}. We consider for $\boldsymbol{\rho}_0\in C^{4+\alpha}_{TJ}(\Gammaast)$, that fulfills the compatibility conditions \eqref{EquationGeometricCompabilityConditionforSDFTJ} and $\|\brho_0\|\le\varepsilon_0$, the solution from Theorem \ref{TheoremSTETripleJunctions}. By choosing $\boldsymbol{\rho}_0$ sufficiently small we can guarantee due to Proposition \ref{PropositionHigherSpaceRegularitySFDTJ} that the solution $\brho$ fulfills for all  $t\in\left[\frac{T_0}{2},T\right]$ 
\begin{align}\label{EquationCOnditionWidetildeTTripleJunction}
\|\brho(t)\|_{C^k_{TJ}(\Gammaast)}\le Z:=\min(\sigma,\varepsilon_0).
\end{align} 
Thus, we can apply (\ref{EquationGeometricLSITripleJunction}) on this set. We now consider the largest $\widetilde{T}$ such that \eqref{EquationCOnditionWidetildeTTripleJunction} is fulfilled on $\widetilde{I}:=[\frac{T}{2},\widetilde{T}]$. Note that due to Theorem \ref{TheoremSTETripleJunctions} the solution $\brho$ exists on this interval. Then, for all $t\in \widetilde{I}$ we have that
\begin{align}
-\frac{d}{dt}(E(\Gamma(t))-E(\Gammaast))^{\theta}&=-\theta(E(\Gamma(t))-E(\Gammaast))^{\theta-1}\sum_{i=1}^3\int_{\Gamma(t)^i}-V_{\Gamma^i(t)}H_{\Gamma^i(t)}\dH^n\notag\\
&=\theta(E(\Gamma(t))-E(\Gammaast))^{\theta-1}\sum_{i=1}^3\int_{\Gamma^i(t)}\left(-\Delta_{\Gamma^i(t)}H_{\Gamma^i(t)}\right)H_{\Gamma^i(t)}\dH^n\notag\\
&=\theta(E(\Gamma(t))-E(\Gammaast))^{\theta-1}\sum_{i=1}^3\int_{\Gamma^i(t)}|\nabla_{\Gamma^i(t)}H_{\Gamma^i(t)}|^2\dH^n\notag\\
&+\int_{\Sigma(t)}\sum_{i=1}^3(\nabla_{\Gamma^i(t)}H_{\Gamma^i(t)}\cdot \nu_{\Gamma^i(t)})H_{\Gamma^i(t)}\dH^{n-1}\notag \\
&\ge C\sqrt{ \sum_{i=1}^3\int_{\Gamma^i(t)}|\nabla_{\Gamma^i(t)}H_{\Gamma^i(t)}|^{^{^2}}\dH^n}\label{EquationStandardEnergyArguemtnLSITripleJunctionI}\\
&= C\|-\Delta_{\Gamma(t)}H_{\Gamma(t)}\|_{\mathcal{H}^{-1}(\Gamma(t))}\notag\\
&\ge C\|-\Delta_{\Gamma(t)}H_{\Gamma(t)}\|_{H^{-1}_{TJ}(\Gamma(t))}\notag\\
&\ge C\|-\Delta_{\Gamma(t)}H_{\Gamma(t)}\|_{H^{-1}_{TJ}(\Gammaast)}\notag\\
&=C\|V_{\Gamma(t)}\|_{H^{-1}_{TJ}(\Gammaast)}\notag\\
&=C\|\partial_t \boldsymbol{\rho}(t)(N_{\boldsymbol{\rho}(t)}\cdot N_{\ast})+\partial_t \boldsymbol{\mu}(t)(N_{\boldsymbol{\rho}(t)}\cdot \tau_{\ast})\|_{H^{-1}_{TJ}(\Gammaast)}\notag\\
&\ge C\Big|\|\partial_t \boldsymbol{\rho}(t)(N_{\boldsymbol{\rho}(t)}\cdot N_{\ast})\|_{H^{-1}_{TJ}(\Gammaast)}-\|\partial_t \boldsymbol{\mu}(t)(N_{\boldsymbol{\rho}(t)}\cdot \tau_{\ast}) \|_{H^{-1}_{TJ}(\Gammaast)}\big|\notag \\
&\ge C\|\partial_t\boldsymbol{\rho}(t)\|_{H^{-1}_{TJ}(\Gammaast)}. \notag
\end{align}  
Hereby, we used in the fifth line the boundary conditions for (SDFTJ) and \eqref{EquationGeometricLSITripleJunction}, in the seventh line equivalence of the $\mathcal{H}^{-1}$ and the $H^{-1}$-products on $\mathcal{E}$, in the eight line the uniform bound \eqref{EquationCOnditionWidetildeTTripleJunction} for $\brho$ and in the eleventh line the inverse triangle inequality.  The last inequality we want to explain in more details. First, we have for any $\boldsymbol{\psi}\in H^1_{TJ}(\Gammaast)$, that
\begin{align}\label{EquationControllingMuI}
\sum_{i=1}^3\int_{\Gammaast^i}\partial_t\rho^i(N_{\brho(t)}^i\cdot \Nast^i)\psi^i\dH^n\ge \min_{\Gammaast}(N_{\brho(t)}\cdot\Nast)\sum_{i=1}^3\int_{\Gammaast^i}\partial_t\rho^i\psi^i\dH^n\ge \frac{1}{2}\sum_{i=1}^3\int_{\Gammaast^i}\partial_t\rho^i\psi^i\dH^n.
\end{align}
Here, we used the fact that by choosing $Z$ in \eqref{EquationCOnditionWidetildeTTripleJunction} small enough we can guarantee that we have $N_{\brho(t)}\cdot N_{\ast}\ge \frac{1}{2}$, using continuous dependence of the normal field on the $C^1$-norm. Consequently, we derive from \eqref{EquationControllingMuI} that
\begin{align}\label{EquationKillingMUI}
\|\partial_t\brho(t)(N_{\brho(t)}\cdot N_{\ast})\|_{H^{-1}_{TJ}(\Gammaast)}\ge \frac{1}{2}\|\partial_t\brho(t)\|_{H^{-1}_{TJ}(\Gammaast)}.
\end{align}
On the other hand, we have for $\boldsymbol{\psi}\in H^1_{TJ}(\Gammaast)$ that
\begin{align}
\sum_{i=1}^3\int_{\Gammaast^i}\partial_t\mu^i(N_{\brho(t)}^i\cdot\tauast^i)\psi^i\dH^n&=\sum_{i=1}^3\int_{\Gammaast^i}\sum_{j=1}^3\left(\partial_t\rho^i\circ\varphi_{ji}\right)(N_{\brho(t)}^i\cdot\tauast^i)\psi^i\dH^n \notag\\&\le C\max_{\Gammaast}(N_{\brho(t)}\cdot\tauast)\sum_{i=1}^3\int_{\Gammaast^i}\partial_t\rho^i\psi^i\dH^n \label{EquationProblemMuII}\\ \notag
&\le \frac{1}{4}\sum_{i=1}^3\int_{\Gammaast^i}\partial_t\rho^i\psi^i\dH^n.
\end{align}
Here, we denote by $\varphi_{ij}$ the identification of $\Gamma^j$ with $\Gamma^i$ given by \eqref{EquationIdentificationoftheGammai}. We used in the second line the transformation formula and in the third step once again that for $\brho$ small enough in the $C^1$-norm we can guarantee that $N_{\brho(t)}\cdot \tau_{\ast}\le \frac{1}{4C}$.
So, we conclude from \eqref{EquationProblemMuII} that
\begin{align}
\|\partial_t \boldsymbol{\mu}(t)(N_{\boldsymbol{\rho}(t)}\cdot \tau_{\ast}) \|_{H^{-1}_{TJ}(\Gammaast)}\le \frac{1}{4} \|\partial_t\brho(t)\|_{H^{-1}_{TJ}(\Gammaast)}. \label{EquationKillingMu}
\end{align}
Combining \eqref{EquationKillingMUI} and \eqref{EquationKillingMu} we get
\begin{align}
\Big|\|\partial_t \boldsymbol{\rho}(t)(N_{\boldsymbol{\rho}(t)}\cdot N_{\ast})\|_{H^{-1}_{TJ}(\Gammaast)}-\|\partial_t \boldsymbol{\mu}(t)(N_{\boldsymbol{\rho}(t)}\cdot \tau_{\ast}) \|_{H^{-1}_{TJ}(\Gammaast)}\big|\ge \frac{1}{4}\|\partial_t\brho(t)\|_{H^{-1}_{TJ}(\Gammaast)}.
\end{align}
This shows the last inequality in \eqref{EquationStandardEnergyArguemtnLSITripleJunctionI}. In total, we get that
\begin{align}\label{EquationStandardEnergyArguemtnLSITripleJunctionII}
\|\partial_t\boldsymbol{\rho}(t)\|_{H^{-1}_{TJ}(\Gammaast)}\le -C\frac{d}{dt}(E(\Gamma(t))-E(\Gammaast))^{\theta}.
\end{align}
Integrating this in time we get for all $t\in \widetilde{I}$ that
\begin{align}\notag
\|\boldsymbol{\rho}(t)\|_{H^{-1}_{TJ}(\Gammaast)}&\le \|\boldsymbol{\rho}(t)-\boldsymbol{\rho}(T_0)\|_{H^{-1}_{TJ}(\Gammaast)}+\|\boldsymbol{\rho}(T_0)\|_{H^{-1}_{TJ}(\Gammaast)}\notag\\
&\le \int_{T_0}^t\|\partial_t \boldsymbol{\rho}\|_{H^{-1}_{TJ}(\Gammaast)}ds+\|\boldsymbol{\rho}(T_0)\|_{H^{-1}_{TJ}(\Gammaast)}\notag\\
&\le-C(E(\Gamma(t))-E(\Gammaast)^{\theta}+C(E(\Gamma(T_0)))-E(\Gammaast))^{\theta}+\|\boldsymbol{\rho}(T_0)\|_{C^{\theta}_{TJ}(\Gammaast)}\label{EquationL2Boundpartialt}\\
&\le C(E(\Gamma(T_0))-E(\Gammaast))^{\theta}+C\|\boldsymbol{\rho}(T_0)\|_{C^{0}_{TJ}(\Gammaast)}^{\theta}\notag\\
&\le C\|\brho(T_0)\|_{C^{2+\alpha}_{TJ}(\Gammaast)}^{\theta}+C\|\brho(T_0)\|_{C^{2+\alpha}_{TJ}(\Gammaast)}^{\theta}\notag \\
&\le C\|\brho_0\|_{C^{2+\alpha}_{TJ}(\Gammaast)}^{\theta}.\notag
\end{align}
Hereby, we used in fourth step interpolation result for H\"older-spaces for the second summand, in the fifth step Lipschitz continuity of the surface area in the $C^1$-norm of the height function and in the last step the bound $R$ for the norm of $\brho$ from Theorem \ref{TheoremSTETripleJunctions}. From this, we get with the interpolation result between $H^{-1}_{TJ}$ and $H^1_{TJ}$ and our smoothing estimate \eqref{EquationEstimateHigherRegularitybrhoTJ} that
\begin{align}\label{EquationWieSollichdichnennen}
\|\brho(t)\|_{L^2_{TJ}(\Gammaast)}\le \|\brho(t)\|_{H^{-1}_{TJ}(\Gammaast)}^{\frac{1}{2}}\|\brho(t)\|_{H^1_{TJ}(\Gammaast)}^{\frac{1}{2}}\le C\|\brho_0\|_{C^{2+\alpha}_{TJ}(\Gammaast)}^{\frac{\theta}{2}}.
\end{align}
Finally, from this we derive for every $t\in\widetilde{I}$ and some $\beta\in(0,1)$ that
\begin{align}\label{EquationOhGoot}
\|\boldsymbol{\rho}(t)\|_{C^k_{TJ}(\Gammaast)}\le  C\|\boldsymbol{\rho}(t)\|_{C^{k'}_{TJ}(\Gammaast)}^{1-\beta}\|\boldsymbol{\rho}(t)\|_{L^2_{TJ}(\Gammaast)}^{\beta}\le C\|\boldsymbol{\rho}_0\|_{C^{2+\alpha}_{TJ}(\Gammaast)}^{\frac{\beta\theta}{2}}\le C\varepsilon^{\frac{\beta\theta}{2}}. 
\end{align}
Here, we used in the first step interpolation results for Besov spaces, in the next step (\ref{EquationL2Boundpartialt}) and Proposition \ref{PropositionHigherSpaceRegularitySFDTJ} and finally the bounds for the initial data. By choosing 
\begin{align}
\varepsilon\le e^{\frac{2}{\beta\theta}\ln\left(\frac{Z}{2C}\right)},
\end{align} 
we get for all $t\in\widetilde{I}$ that
\begin{align}\label{EquationAbschatzungBliBlaBlub}
\|\brho(t)\|_{C^k_{TJ}(\Gammaast)}\le \frac{Z}{2}.
\end{align}
Applying close to $\widetilde{T}$ our short time existence Theorem \ref{TheoremSTETripleJunctions} and using again Proposition \ref{PropositionHigherSpaceRegularitySFDTJ} we see that $\widetilde{T}$ cannot be maximal such that (\ref{EquationCOnditionWidetildeTTripleJunction}) is fulfilled. This shows now that $\widetilde{T}=\infty$ and consequently global existence of $\brho$.\\
Next, we want to see that $\brho$ not only exists globally in time but also converges for $t\to\infty$ to an energy minimizer. First, we observe that as $E(\Gamma(t))-E(\Gammaast)$ is strictly decreasing while bounded below by $0$, the right hand side in \eqref{EquationStandardEnergyArguemtnLSITripleJunctionII} is in $L^1(\R^+, \R)$. Consequently, we see that $\partial_t\brho\in L^1(\R^+,H^{-1}_{TJ}(\Gammaast))$ and therefore there is $\brho_{\infty}\in H^{-1}_{TJ}(\Gammaast)$ with $\brho(t)\to\brho_{\infty}$ for $t\to\infty$. Using the same interpolation arguments as in \eqref{EquationWieSollichdichnennen} and \eqref{EquationOhGoot} we see that the convergence holds also in $C^{k}_{TJ}(\Gammaast)$
for any $k>0$. Thus, the limit $\brho_{\infty}$ is smooth enough such that we can apply the LSI in the corresponding limit surface $\Gamma_{\infty}$, deriving that
\begin{align*}
|E(\Gamma_{\infty})-E(\Gammaast)|^{\theta}\le C\|-\Delta_{\Gamma_{\infty}}H_{\Gamma_{\infty}}\|_{\mathcal{H}^{-1}(\Gamma_{\infty})}=0.
\end{align*}
In the last step we used that $\partial_t\brho(t)\in L^1(\R^+, H^{-1}_{TJ}(\Gammaast))$, which is only possible for $\|\partial_t\brho\|_{H^{-1}(\Gammaast)}\to 0$. This shows now that $\Gamma_{\infty}$ is an energy minimizer, i.e., a standard double bubble.\\
It remains to prove stability in the sense of Lyapunov. So, let $\varepsilon>0$. To control $\brho(t)$ we have to differ two cases. Either, we have $t\in[0,T_0]$. On this interval, we can use the estimates from the analysis for the short time existence, c.f. \cite[Section 6,(107)]{garckegoesweinpreprintshorttimeexistenceSDFTJ}, to see that for $\brho_0$ and $T_0$ small enough we have that $\|\brho(t)\|_{C^{4+\alpha}_{TJ}(\Gammaast)}\le \varepsilon$ on $[0, T_0]$. On the interval $[T_0,\infty)$ we can use the estimate \eqref{EquationAbschatzungBliBlaBlub}. Observe that it will stay true for any $0<C\le Z$ as long as we can guarantee that $\|\brho(T_0)\|_{C^{k}_{TJ}(\Gammaast)}\le C$. Using parabolic smoothing of the flow this will be fulfilled after choosing $\brho_0$ possibly smaller. So together we see that provided $\brho_0$ is small enough in the $C^{4+\alpha}$-norm, we will have that $\|\brho(t)\|_{C^{4+\alpha}_{TJ}(\Gammaast)}\le \varepsilon$ on $\R^+$. This implies stability in the sense of Lyapunov and finishes the proof of Theorem \ref{TheoremStabilityofStationarydoublebubbles}.

\section*{Acknowledgement}
The second author was partially funded by the DFG
through the Research Training Group GRK 1692 \textit{Curvature, Cycles, and Cohomology} in Regensburg. The support is gratefully acknowledged.
\bibliographystyle{acm}
\bibliography{BibFileSTESurfDiffTripJunc} 

\end{document}